\documentclass[11pt,reqno]{amsart}

\usepackage[a4paper,textwidth=6.5in,top=25mm,bottom=25mm,headheight=12pt,headsep=12pt]{geometry}
\usepackage{comment}
\usepackage{amsmath}
\usepackage{amsfonts}
\usepackage{amssymb}
\usepackage{amsthm}
\usepackage{esint}
\usepackage{mathtools}
\usepackage{array,makecell,longtable}
\usepackage{xcolor}
\usepackage{graphicx}
\usepackage[toc,title,page]{appendix}
\usepackage{subcaption}
\usepackage{cite}
\usepackage{etoolbox}
\usepackage{color}
\usepackage{bm}
\usepackage{mathrsfs}
\usepackage{enumitem}
\usepackage{float}
\usepackage{microtype}

\makeatletter
\def\hyper@nopatch@thm{}
\makeatother
\usepackage[colorlinks=true,citecolor=blue,linkcolor=black,urlcolor=black]{hyperref}

\setlist{topsep=4pt plus 1pt minus 1pt,itemsep=2pt,parsep=0pt}
\theoremstyle{remark}
\newtheorem{remark}{Remark}[section]
\theoremstyle{plain}
\newtheorem{theorem}[remark]{Theorem}

\newtheorem{lemma}[remark]{Lemma}

\theoremstyle{definition}
\newtheorem{definition}[remark]{Definition}

\theoremstyle{plain}
\numberwithin{equation}{section}
\allowdisplaybreaks[2]

\title[Fourier--Laplace Injectivity for Oblique and Polygonal Cylinders in Passive IPs]{Fourier--Laplace Injectivity for Oblique and Polygonal Cylinders in Passive Inverse Problems}

\author{Yan Jiang}
\address{Department of Mathematics, City University of Hong Kong, Hong Kong SAR, China.}
\email{yjian24@cityu.edu.hk}

\author{Hongyu Liu}
\address{Department of Mathematics, City University of Hong Kong, Hong Kong SAR, China.}
\email{hongyu.liuip@gmail.com, hongyliu@cityu.edu.hk}

\author{Tianhao Ni}
\address{School of Mathematical Sciences, Fudan University, Shanghai, China.}
\email{thni@fudan.edu.cn}

\author{Kai Zhang}
\address{Department of Mathematics, Jilin University, Changchun, Jilin, China.}
\email{zhangkai7609@gmail.com, zhangkaimath@jlu.edu.cn}

\keywords{Fourier--Laplace transform; totally real submanifold; Brion--Barvinok theorem; passive inverse problem; complex geometrical optics}
\subjclass[2020]{35R30, 35P25}
\date{}

\hypersetup{
  hypertexnames=false,
  pdfauthor={Yan Jiang, Hongyu Liu, Tianhao Ni, and Kai Zhang},
  pdftitle={Fourier--Laplace Injectivity for Oblique and Polygonal Cylinders in Passive Inverse Problems}
}

\begin{document}
\raggedbottom
\setlength{\abovedisplayskip}{7pt plus 2pt minus 2pt}
\setlength{\belowdisplayskip}{7pt plus 2pt minus 2pt}
\setlength{\abovedisplayshortskip}{0pt plus 2pt}
\setlength{\belowdisplayshortskip}{4pt plus 2pt minus 2pt}

\begin{abstract}

We study Fourier--Laplace transforms on null phases arising in passive
inverse problems. Densities with codimension-one dependence in oblique cylinders are first considered. Integration in the distinguished
direction reduces the data to an entire transform on a real-analytic maximally
totally real submanifold. Holomorphic uniqueness then gives injectivity. A fixed transverse polytope swept along a polygonal curve, forming a polygonal cylinder, is next considered. The
density is assumed to be polynomial in the material transverse variables on
each swept piece. A weighted Brion--Barvinok expansion separates the vertex
terms and the successive pieces. The same representation gives a
finite-dimensional Lipschitz estimate when the polynomial degree is fixed.

We apply these transform results to passive wave and heat equations in
arbitrary dimension. The initial state is zero outside the observation domain.
Equality of the Laplace-domain boundary data determines the quantity of the
initial data divided by the speed squared in the acoustic model and the
quantity of the initial data divided by the coefficient of thermal
conductivity in the parabolic model. The question of what
remains true when only weighted generalized boundary data on a finite interval
\([0,T]\) are available is then studied. These data determine the harmonic
moments of the quotient up to the weighted terminal residual. This identity
gives a H\"older estimate in an oblique cylinder and a Lipschitz estimate in a
polygonal cylinder. Finally, a continuous generating curve is approximated by a
polygonal curve and a continuous material density by piecewise transverse
polynomials. The resulting reconstruction estimate separates measurement
noise, density approximation error, geometric error, and the condition number
of the finite sampling matrix.

\end{abstract}

\maketitle

%%%%%%%%%%%%%%%%%%%%%%%%%%%%%%%%%%%%%%%%%%%%%%%%%%%%%%%%%%%%

\section{Introduction}

Passive inverse problems arise when the available measurements are generated by
an unknown initial state rather than by boundary inputs that can be freely
prescribed. This situation is natural in thermoacoustic and photoacoustic
tomography, where an electromagnetic excitation produces an acoustic wave whose
initial pressure is unknown and whose propagation is affected by the acoustic
properties of the medium
\cite{kruger1999thermoacoustic,beard2011biomedical,
wang2012photoacoustic,poudel2019survey,matthews2018parameterized,
clason2008quasi,klibanov2008new,klibanov2013thermo}.
In the standard mathematical model, the measured wave field is generated by a
single initial state. Therefore, in contrast to active boundary value problems,
the data do not allow one to prescribe two independent boundary inputs or to form
arbitrary products of solutions. This lack of independent probing is one of the
main obstructions in simultaneous recovery problems for the sound speed and the
initial pressure
\cite{liu2015determining,knox2020determining,kian2025determination,chen2025determining,kian2026uniqueness}.

Biomedical imaging provides a natural setting for these questions, from
thermoacoustic and photoacoustic tomography (TAT/PAT) to functional photoacoustic
brain imaging
\cite{beard2011biomedical,na2022brain}.
Related electromagnetic problems arise in brain imaging, where currents and material parameters are sought from boundary field measurements
\cite{deng2019brain},
and in Wi-Fi sensing, where received radio signals are used to monitor human motion and respiration
\cite{tan2015wifi,armenta2024wifi}.
Recent studies of passive recovery under structural assumptions concern wave equations
\cite{kian2026uniqueness},
subdiffusion
\cite{deng2026subdiffusion},
nonlinear balance laws
\cite{duan2025balance}
and structured elastic media
\cite{gao2026passive}.

This restriction on independent probing distinguishes passive problems from the classical active inverse
boundary value problems. For instance, in the Calder\'on problem and its
Schr\"odinger formulation, equality of Dirichlet-to-Neumann maps yields an
identity of the form
\begin{equation*}
\int_\Omega (q_1-q_2)u_1u_2\,dx=0,
\end{equation*}
where \(q_1\) and \(q_2\) are two candidate targets for the unknown
object quantity, the boundary values of \(u_1\) are the measurements,
and \(u_2\) is the test function. Complex geometrical optics solutions then reduce the identity to the vanishing of the Fourier transform of \(q_1-q_2\). This mechanism underlies the classical uniqueness results for the Calder\'on problem and related Schr\"odinger inverse problems \cite{calderon2006inverse,sylvester1987global,nachman1996global,uhlmann2009electrical}. Similar ideas also appear in many hyperbolic and nonlinear inverse problems,
where geometric optics, Gaussian beams, Runge approximation, or higher order
linearization generate sufficiently rich weighted integral identities
\cite{kurylev2018inverse,lassas2018inverse,hintz2022inverse,
uhlmann2021inverse,jiang2025inverse1,lin2024determining,jiang2025inverse2,
kian2019recovery,feizmohammadi2021recovery,lin2022simultaneous,ding2025stationary,liu2025running,liu2025decoding,fu2026kirchhoff}.
The common feature in these active settings is that the measurement map provides
enough freedom to produce a large family of testing weights.
For parabolic equations, simultaneous recovery of a space--time dependent
coefficient and initial conditions from a one-parameter family of lateral
Cauchy data is studied in \cite{klibanov2026simultaneous}.

In the passive setting, Green's identity leaves only one adjustable test
solution. Such data are not injective on general \(n\)-dimensional densities. We identify two codimension-one classes for which injectivity is recovered.

For the first problem, we consider a class of generalized oblique cylinders, comprising oblique cylinders and finite unions of parallel layers. For
densities depending only on the transverse variable, integration along the
longitudinal direction produces an entire Fourier--Laplace transform on a
real-analytic, maximally totally real submanifold of \(\mathbb{C}^{n-1}\). The standard holomorphic uniqueness principle on such submanifolds then gives injectivity; see \cite{harris1983real,baouendi1999real}. 

For the second problem, the domain is obtained by sweeping a fixed transverse polytope along a polygonal curve; we refer to the resulting domain as a polygonal cylinder. For densities that are
polynomial in the material transverse variables on each straight piece, the weighted Brion--Barvinok formula derived from \cite{barvinok1992exponential,barvinok1993computing} gives a finite vertex expansion. Fourier--Laplace transforms of polytopes on rationally parametrized sets, and the related Pompeiu problem, have been studied in \cite{engel2023identity,machado2023null,kolountzakis2024curves}. Polynomial
densities on polytopes also occur in the moment problem; see
\cite{gravin2018moments}. The ordering of the longitudinal endpoints
separates the swept pieces, while generic transverse directions distinguish
the vertex phases. For a fixed polynomial degree, the same construction
yields a conditional finite-dimensional stability estimate. The null phases
used here form a zero-radius isotropic cone. The general identity theorems
in \cite{engel2023identity,kolountzakis2024curves} do not directly treat the
ordered, polynomially weighted sweep considered below. Our proof uses the
longitudinal order to isolate one endpoint section at a time.

The finite-dimensional estimate also provides a controlled approximation method
for a continuous generating curve and a continuous density. The curve is replaced by
its piecewise linear interpolant. On each segment, the density is replaced by a
polynomial in the material transverse variables. The estimate keeps the two
model errors separate. It also shows why polynomial approximation alone does not
imply convergence: the inverse of the sampling matrix may become large as the
number of segments or the polynomial degree increases.

For the PDE applications, the oblique cylinder section is specialized to an
interval, while the polygonal cylinder geometry is retained. The initial state
and the coefficient perturbation are sharply truncated to the observation
domain and may jump across its boundary. The structural assumption is imposed
on the compactly supported quotient. The Laplace-domain boundary data determine this quotient in arbitrary dimension under either geometric assumption and the stated local decay conditions. This low-frequency step extends the separated-variable quotient recovery in \cite{chen2025determining}. The new points are the translated-section geometry, the polygonal sweep, the sharp cutoff, and the corresponding transform arguments.

When the observation stops at a finite time \(T\), the proof no longer yields
the vanishing transform of the quotient. Instead, the generalized boundary
data of the time-integrated state give an endpoint identity. If the two
finite-time measurements agree, the transform of the quotient difference
equals the transform of the terminal residual. We quantify this loss in each of the two problems. Thus finite-time equality gives exact uniqueness only when the terminal residual vanishes. A decay estimate gives approximate uniqueness as \(T\) increases. These statements describe the loss caused by truncating the observation time. They are not estimates for noisy boundary measurements.

The paper is organized as follows. Section~\ref{sec:main_results} introduces
the oblique cylinder and polygonal cylinder and states the transform
injectivity, wave and heat quotient-recovery, and finite-time results.
Section~\ref{sec:totally_real_core} proves the oblique cylinder result by
separating the longitudinal variable and applying holomorphic uniqueness on a
maximally totally real submanifold. Section~\ref{sec:brion_barvinok_core}
treats the polygonal cylinder: it derives the weighted Brion--Barvinok
expansion, proves endpoint separation, establishes finite-dimensional
stability, and studies polygonal and polynomial approximation of a continuous
sweep. Section~\ref{sec:hyperbolic_quotient_recovery} constructs the localized
elliptic corrector and applies the low-frequency Laplace-domain identity to
the wave equation. Section~\ref{sec:parabolic_application} gives the analogous
argument for the heat equation. Section~\ref{sec:finite_time_stability}
introduces the weighted generalized finite-time boundary data, proves the
endpoint identity, and derives the H\"older and Lipschitz estimates for both
problems. Section~\ref{sec:conclusion} summarizes the results and their
limitations.

%%%%%%%%%%%%%%%%%%%%%%%%%%%%%%%%%%%%%%%%%%%%%%%%%%%%%%%%%%%%%%%%%%%%%%
\section{Mathematical Setting and Main Results}\label{sec:main_results}

We introduce the two geometries used in the Fourier--Laplace injectivity
results and then formulate the corresponding passive wave and heat problems.
For a measurable set \(A\), its characteristic function is denoted by \(\chi_A\).

%%%%%%%%%%%%%%%%%%%%%%%%%%%%%

\subsection{Oblique cylinder geometry}
\label{subsec:Oblique_cylinder_geometry}

Fix \(n\ge2\) and a unit vector \(\theta\in\mathbb S^{n-1}\), and set
\(\theta^\perp=\{y\in\mathbb R^n:y\cdot\theta=0\}\).
Identify \(\theta^\perp\) with \(\mathbb R^{n-1}\). Let \(D\subset\theta^\perp\) and \(E\subset\mathbb R\) be bounded measurable sets of positive measure, and let \(\mu\in\theta^\perp\) and \(x_0\in\mathbb R^n\).
Define
\begin{equation}\label{eq:general_layer_domain}
\Omega_{\theta,\mu,E}:=\{x=x_0+y+s\theta:y\in D,~s\in \mathbb R,~s-\mu\cdot y\in E\}.
\end{equation}

For fixed \(y\in D\), the section of
\(\Omega_{\theta,\mu,E}\) at \(y\) is defined by
\begin{equation*}
(\Omega_{\theta,\mu,E})_y:=\{s\in\mathbb R: x_0+y+s\theta\in\Omega_{\theta,\mu,E}\}.
\end{equation*}
Since \(x_0+y+s\theta\in\Omega_{\theta,\mu,E}\) if and only if
\(s-\mu\cdot y\in E\), we obtain
\begin{equation*}
(\Omega_{\theta,\mu,E})_y=\mu\cdot y+E.
\end{equation*}

For a set \(E\subset\mathbb R\) and \(a\in\mathbb R\), the translate
of \(E\) by \(a\) is
\begin{equation*}
E+a:=\{t+a:t\in E\}.
\end{equation*}
Thus, taking \(a=\mu\cdot y\), each section is the translate
\(E+\mu\cdot y\) of the fixed set \(E\). Hence all sections are
translates of \(E\).

In three dimensions, take \(\theta=e_3=(0,0,1)\), \(x_0=0\), and
\(E=(0,h)\), with \(D\subset\{x_3=0\}\) and
\(\mu=(\mu_1,\mu_2,0)\). This yields a cylinder whose generators are
parallel to the \(x_3\)-axis. When \(\mu\neq 0\), its two end sections lie in
the parallel oblique planes
\begin{equation*}
x_3=\mu_1x_1+\mu_2x_2
\quad\text{and}\quad
x_3=\mu_1x_1+\mu_2x_2+h,
\end{equation*}
so the resulting cylinder is an oblique cylinder of the type considered below.

Let $U\Subset\theta^\perp\setminus\{0\}$ be a nonempty open set. For
$\eta\in U$, define
\begin{equation*}
\zeta_\theta(\eta):=\eta+i|\eta|\theta\in\mathbb C^n.
\end{equation*}
Then
\begin{equation*}
\zeta_\theta(\eta)\cdot\zeta_\theta(\eta)=|\eta|^2+(i|\eta|)^2=0.
\end{equation*}

The translated-section structure separates the longitudinal variable and
leads to the following injectivity result.
\begin{theorem}
\label{thm:main_totally_real_uniqueness}
Let $\Omega_{\theta,\mu,E}$ be given by
\eqref{eq:general_layer_domain} and
\begin{equation*}
G(x)=G_0(y)\chi_{\Omega_{\theta,\mu,E}}(x), \quad G_0\in L^1(D),
\end{equation*}
where $x=x_0+y+s\theta$. Assume that
\begin{equation}\label{eq:main_TR_identity_minus}
\int_{\Omega_{\theta,\mu,E}}G(x)e^{-ix\cdot\zeta_\theta(\eta)}\,dx=0,
\quad \eta\in U. 
\end{equation}
Then
\begin{equation*}
G=0, \quad \text{a.e. }\Omega_{\theta,\mu,E}.
\end{equation*}
\end{theorem}
The conclusion still holds with \eqref{eq:main_TR_identity_minus} replaced by
\begin{equation*}
\int_{\Omega_{\theta,\mu,E}}
G(x)e^{ix\cdot\zeta_\theta(\eta)}\,dx=0,
\qquad \eta\in U.
\end{equation*}	
The proof is given in Section~\ref{sec:totally_real_core}. The PDE
applications use the interval case $E=(0,h)$ for $h>0$.

%%%%%%%%%%%%%%%%%%%%%%%%%%%%%%%%%%%%%%%%%%%%%%%%%%%%%%%%%%%%%

\subsection{Polygonal cylinder with a common transverse base}
\label{subsec:Polygonal_cylinder_geometry}

We next introduce the second injectivity geometry. Let
$D\subset\theta^\perp$ be a compact simple convex polytope
with nonempty interior. Thus every vertex of $D$ is generated by exactly
$n-1$ linearly independent inward edge directions. Fix
\begin{equation*}
h_j>0, \quad \gamma_j\in\theta^\perp, \quad j=1,\cdots,N,
\end{equation*}
and define the cumulative longitudinal and transverse displacements
\begin{equation}\label{eq:polygonal_cumulative_data}
t_0:=0,\quad t_j:=\sum_{k=1}^j h_k, \quad b_0:=0,\quad b_j:=\sum_{k=1}^j h_k\gamma_k .
\end{equation}
The polygonal generatrix is the continuous piecewise-affine curve
\begin{equation*}
\Gamma(t):=b_{j-1}+(t-t_{j-1})\gamma_j+t\theta,\quad t_{j-1}\leq t\leq t_j.
\end{equation*}
We write $\Gamma_j:=\Gamma([t_{j-1},t_j])$ for the image of its $j$-th
straight segment. In particular,
\begin{equation*}
\Gamma(t)\cdot\theta=t,\quad r_j:=\Gamma(t_j)=b_j+t_j\theta .
\end{equation*}
The same transverse base $D$ is translated along $\Gamma$ by the sweep map
\begin{equation*}
\Phi:\operatorname{int}D\times(0,t_N)\longrightarrow\mathbb R^n,
\quad \Phi(y,t):=x_0+y+\Gamma(t),
\end{equation*}
where $\operatorname{int}D$ denotes the interior of \(D\). Since the  $L^2$-projection of $(\Phi(y,t)-x_0)$ onto $\theta$ ($\theta$-section) equals $t$, the map $\Phi$ is injective and has the explicit piecewise-Lipschitz inverse
\begin{equation*}
t=(x-x_0)\cdot\theta,\quad y=\Pi_{\theta^\perp}(x-x_0)-b(t),
\end{equation*}
where $b(t)=b_{j-1}+(t-t_{j-1})\gamma_j$ on $[t_{j-1},t_j]$ and $\Pi_{\theta^\perp}$ denotes the  $L^2$-projection onto $\theta^\perp$. Thus $\Phi$ is bi-Lipschitz onto its image. We write
\begin{equation*}
\mathcal B:=\bigl(D;\,(\gamma_1,h_1),\cdots,(\gamma_N,h_N)\bigr),
\end{equation*}
for these ordered polygonal cylinder data.

We use a half-open convention only to assign each internal section to exactly
one segment. For $j=1,\cdots,N-1$, set
\begin{equation*}
P_j:=\{\Phi(y,t):y\in\operatorname{int}D,~t_{j-1}<t\leq t_j\},
\end{equation*}
and for the final segment set
\begin{equation*}
P_N:=\{\Phi(y,t):y\in\operatorname{int}D,~t_{N-1}<t<t_N\}.
\end{equation*}
Consequently,
\begin{equation}\label{eq:broken_layer_domain_main}
\Omega_{\mathcal B}:=\bigcup_{j=1}^N P_j=\Phi\bigl(\operatorname{int}D\times(0,t_N)\bigr).
\end{equation}
The pieces are pairwise disjoint, while their union is an open, connected,
bounded Lipschitz domain. Changing the assignment of the internal sections
$t=t_j$ does not affect any integral identity. Figure~\ref{fig:polygonal_sweep_cylinder} shows the construction in three
dimensions.
\begin{figure}[t]
\centering
\includegraphics[
width=0.72\textwidth,
trim=30 90 30 15,
clip
]{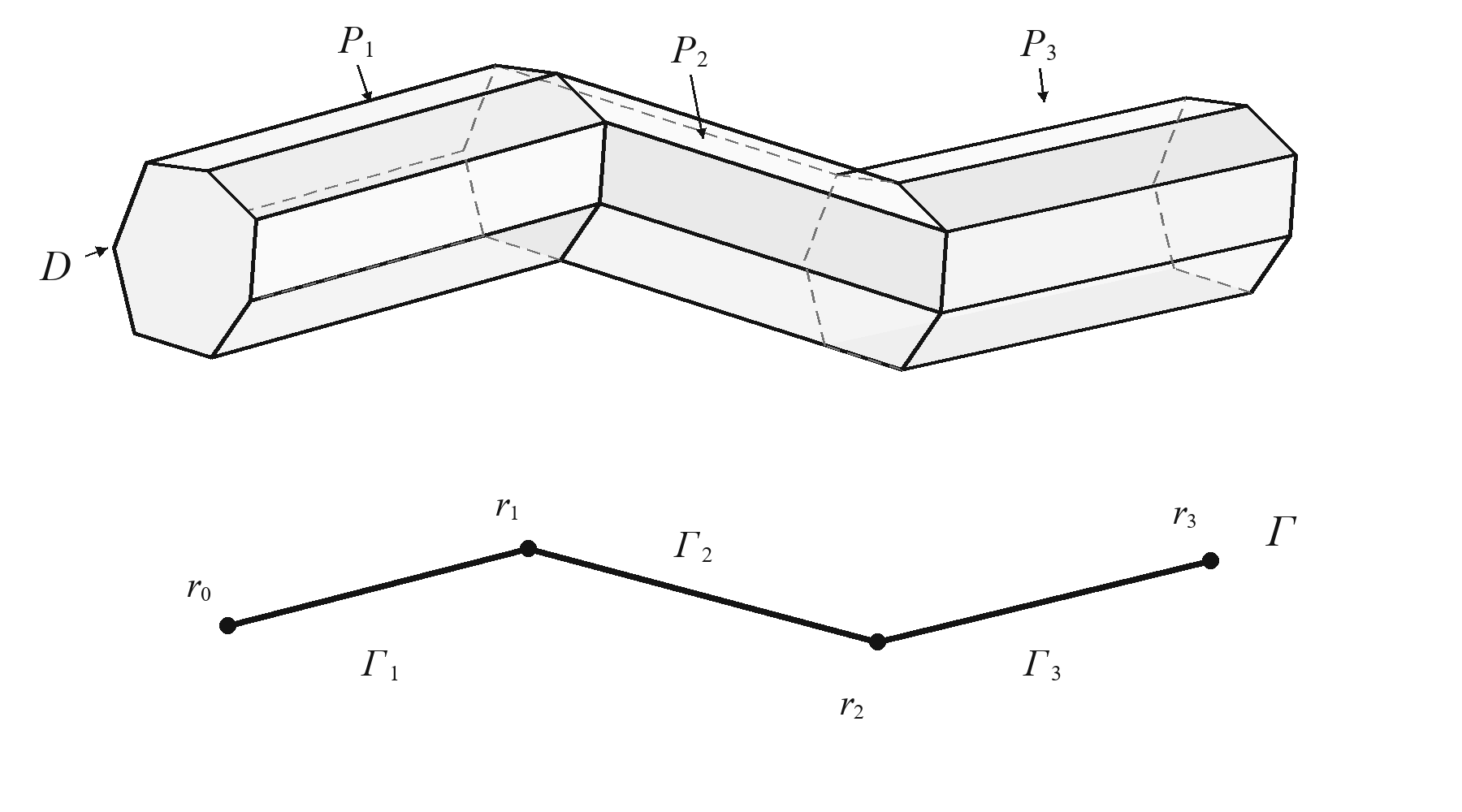}
\caption{Three-dimensional schematic ($n=3$): the four congruent section contours are
translates of the base $D$. The swept pieces are $P_1,P_2,P_3$. The detached
polygonal generatrix is $\Gamma=\Gamma_1\cup\Gamma_2\cup\Gamma_3$, with
$r_j=\Gamma(t_j)$. Solid and dashed lines distinguish visible and hidden
edges; internal sections follow the half-open convention.}
\label{fig:polygonal_sweep_cylinder}
\end{figure}

Let $V(D)$ denote the vertex set of $D$. For
$\omega\in\theta^\perp\cap\mathbb S^{n-1}$ and $v\in V(D)$, define the
complex lower and upper endpoint phases
\begin{equation}\label{eq:lower_upper_height_main}
L_j(v;\omega):=t_{j-1}-i\omega\cdot(b_{j-1}+v),\quad
U_j(v;\omega):=t_j-i\omega\cdot(b_j+v).
\end{equation}
Their real parts are
\begin{equation*}
\operatorname{Re}L_j(v;\omega)=t_{j-1},\quad
\operatorname{Re}U_j(v;\omega)=t_j.
\end{equation*}
Hence the largest upper endpoint among any nonempty collection of segments
belongs to the segment with largest index, and the smallest lower endpoint
belongs to the segment with smallest index. Vertices on the same endpoint
section have equal real part but are separated by their oscillatory phases for
a generic direction $\omega$.

Let $Y:\Omega_{\mathcal B}\to\operatorname{int}D$ denote the transverse
coordinate determined by
\begin{equation*}
\Phi^{-1}(x)=(Y(x),T(x)).
\end{equation*}
The polygonal cylinder geometry leads to the following injectivity statement.

\begin{theorem}\label{thm:main_brion_barvinok_uniqueness}
Let $\Omega_{\mathcal B}$ be given by \eqref{eq:broken_layer_domain_main}, and let
\begin{equation*}
G(x)=\sum_{j=1}^N p_j(Y(x))\chi_{P_j}(x),
\end{equation*}
where each $p_j\in\mathbb C[y_1,\ldots,y_{n-1}]$. If either
\begin{equation}\label{eq:BB_minus_identity_main}
\int_{\Omega_{\mathcal B}} G(x)e^{-ix\cdot\zeta_\theta(\eta)}\,dx=0
\quad\text{or}\quad
\int_{\Omega_{\mathcal B}} G(x)e^{ix\cdot\zeta_\theta(\eta)}\,dx=0,
\quad \eta\in U,
\end{equation}
for a nonempty open $U\Subset\theta^\perp\setminus\{0\}$, then
$p_1\equiv\cdots\equiv p_N\equiv0$, and hence $G=0$ a.e. in
$\Omega_{\mathcal B}$.
\end{theorem}
The proof is given in Section~\ref{sec:brion_barvinok_core}. The role of the polynomial assumption is discussed in Remark~\ref{rem:polynomial_not_analytic}.

%%%%%%%%%%%%%%%%%%%%%%%%%%%%%%%%%%%%%%%%%%%%%%%%%%%%%%%%%%%%%	

\subsection{Oblique and polygonal cylinders for PDE applications}
\label{subsec:pde_geometry}

For hyperbolic and parabolic inverse problems, we use the interval case of the
oblique cylinder geometry, allowing an arbitrary distinguished direction. 
Let $E=(0,h)$ in \eqref{eq:general_layer_domain}. Thus every $\theta$-section is the interval $(\mu\cdot y,\mu\cdot y+h)$. If \(D\) is Lipschitz, then
\(\partial\Omega_{\theta,\mu,h}\) is Lipschitz, so the usual trace
operations used below are available on this observation domain. For \(x=x_0+y+s\theta\),
\begin{equation*}
e^{-ix\cdot\zeta_\theta(\xi)}
=e^{-ix_0\cdot\zeta_\theta(\xi)}e^{-i\xi\cdot y}e^{|\xi|s}.
\end{equation*}
Hence this geometry reduces to
Theorem~\ref{thm:main_totally_real_uniqueness}.	

The PDE applications below use either of two observation geometries:
\begin{equation}\label{eq:Omega_sharp_alternatives}
\Omega_\sharp=\Omega_{\theta,\mu,h} \quad\text{or}\quad \Omega_\sharp=\Omega_{\mathcal B}.
\end{equation}
The first is the oblique cylinder \eqref{eq:general_layer_domain} with \(E=(0,h)\); the second is the bounded Lipschitz polygonal cylinder \eqref{eq:broken_layer_domain_main}. The negative phase
\(e^{-ix\cdot\zeta_\theta(\xi)}\), produced by the low-frequency limit in the
hyperbolic and parabolic proofs below, peels off its segments in decreasing
order of their terminal parameters \(t_j\).

Throughout the PDE sections, fix
\begin{equation}\label{eq:pn_condition}
p_n\in(1,\infty),\quad p_n>\max\left\{\frac n2,2\right\}.
\end{equation}
This exponent is used only in the localized elliptic correction, via the
embedding \(W^{2,p_n}(Q)\hookrightarrow L^\infty(Q)\), to allow the whole-space
coefficient to have the jump created by the sharp truncation.

%%%%%%%%%%%%%%%%%%%%%%%%%%%%%%%%%%%%%%%%%%	

\subsection{Hyperbolic passive inverse problem}\label{subsec:main_hyperbolic_passive}
We consider the whole-space acoustic wave equation
\begin{equation}
\begin{cases}
q(x)\partial_t^2u(t,x)-\Delta u(t,x)=0,
& (t,x)\in(0,\infty)\times\mathbb R^n,\\
u(0,x)=f(x),\quad \partial_tu(0,x)=0,
& x\in\mathbb R^n,
\end{cases}
\label{eq:main_hyperbolic_forward}
\end{equation}
where \(q=c^{-2}\). Since the sharp truncation may create jumps, the
time-domain Cauchy trace need not be defined. We therefore use the
Laplace-domain family as the observation. For \(\tau>0\), set
\begin{equation*}
U_\tau(x):=\int_0^\infty e^{-\tau t}u(t,x)\,dt .
\end{equation*}
Then \(U_\tau\in H^2(\mathbb R^n)\) and
\begin{equation}\label{eq:main_hyperbolic_laplace_equation}
(-\Delta+\tau^2q)U_\tau=\tau qf, \quad\text{in }\mathbb R^n.
\end{equation}
The passive Cauchy observation from the single time trace is the family
\begin{equation}\label{eq:main_hyperbolic_laplace_data}
\Lambda_{f,c}^{\Omega_\sharp}(\tau):=(\gamma_DU_\tau,\gamma_NU_\tau),
\quad \tau>0.
\end{equation}
Here and below, boundary pairs lie in $H^{1/2}(\partial\Omega_\sharp)
\times H^{-1/2}(\partial\Omega_\sharp)$. The Dirichlet trace
\(\gamma_DU_\tau\in H^{1/2}(\partial\Omega_\sharp)\) is the usual one.
The weak Neumann trace
\(\gamma_NU_\tau\in H^{-1/2}(\partial\Omega_\sharp)\) is defined by
\begin{equation*}
\langle\gamma_NU_\tau,\gamma_D\varphi\rangle_{\partial\Omega_\sharp}
=\int_{\Omega_\sharp}\bigl(\nabla U_\tau\cdot\nabla\varphi
+(\Delta U_\tau)\varphi\bigr)\,dx,\quad \varphi\in H^1(\Omega_\sharp).
\end{equation*}
Indeed, if \(\gamma_D\varphi=0\), then \(\varphi\in H_0^1(\Omega_\sharp)\), and
approximation by \(C_c^\infty(\Omega_\sharp)\) together with the distributional definition of \(\Delta U_\tau\) makes the integral vanish.

%%%%%%%%%%%%%

\begin{definition}\label{def:hyperbolic_admissibility}
Let \(\Omega_\sharp\) be either domain in \eqref{eq:Omega_sharp_alternatives}. A pair \((f,c)\) is \emph{hyperbolically admissible} relative to \(\Omega_\sharp\) if the following hold.
\begin{enumerate}
\item \textbf{Sharp truncation and positivity.}
There exist \(f_{\mathrm{in}}\in L^2(\Omega_\sharp)\), \(c_{\mathrm{in}}\in L^\infty(\Omega_\sharp)\), and \(0<c_-\le c_+<\infty\) such that
\begin{equation}\label{eq:main_hyperbolic_common_truncation}
f=f_{\mathrm{in}}\chi_{\Omega_\sharp},\quad
c=1+(c_{\mathrm{in}}-1)\chi_{\Omega_\sharp},\quad
c_-\le c_{\mathrm{in}}\le c_+,\quad \text{a.e. in }\Omega_\sharp.
\end{equation}
Thus \(f\) and \(c-1\) share the same sharp cutoff at \(\partial\Omega_\sharp\),
\begin{equation*}
\operatorname*{ess\,supp}f\cup\operatorname*{ess\,supp}(c-1)
\subset\overline{\Omega}_\sharp,
\end{equation*}
and jumps across \(\partial\Omega_\sharp\) are allowed. Equivalently,
\(q\in L^\infty(\mathbb R^n)\), \(0<q_-\le q\le q_+<\infty\), and
\(q=1\) on \(\mathbb R^n\setminus\overline{\Omega_\sharp}\).

\item \textbf{Transverse quotient structure.}
If \(\Omega_\sharp=\Omega_{\theta,\mu,h}\), then there exists
\(F\in L^1(D)\cap L^2(D)\) such that
\begin{equation}\label{eq:main_hyperbolic_quotient_structure_translated}
\frac{f_{\mathrm{in}}(x)}{c_{\mathrm{in}}^2(x)}=F(y),
\quad\text{a.e. }\Omega_{\theta,\mu,h}.
\end{equation}
Hence \(f/c^2=F(y)\chi_{\Omega_{\theta,\mu,h}}\).

If \(\Omega_\sharp=\Omega_{\mathcal B}=\bigcup_{j=1}^N P_j\), then there exist
transverse polynomials \(F_j\in\mathbb C[y_1,\dots,y_m]\) such that
\begin{equation}\label{eq:main_hyperbolic_quotient_structure_broken}
\frac{f_{\mathrm{in}}(x)}{c_{\mathrm{in}}^2(x)}
=\sum_{j=1}^N F_j(Y(x))\chi_{P_j}(x),\quad\text{a.e. } \Omega_{\mathcal B}.
\end{equation}
Thus \(f/c^2=\sum_{j=1}^N F_j(Y(x))\chi_{P_j}(x)\), compactly supported, with
jumps allowed at the outer boundary and internal polygonal sections.

\item \textbf{Mild forward solution and local decay.}
Let \(u\in C([0,\infty);L^2(\mathbb R^n))\) be the mild solution of
\eqref{eq:main_hyperbolic_forward}. Assume there exist \(C>0\) and a
nonnegative measurable \(\eta\in L^1(0,\infty)\) with \(\eta(t)\to0\) as
\(t\to\infty\), such that
\begin{equation*}
\|u(t,\cdot)\|_{L^2(\Omega_\sharp)}\le C\eta(t)\|f\|_{L^2(\mathbb R^n)},\quad t>0.
\end{equation*}
\end{enumerate}
\end{definition}

%%%

\begin{remark}\label{rem:hyperbolic_forward_wellposedness}
Let \(A_q\) be the nonnegative self-adjoint operator in
\(L^2(\mathbb R^n,q\,dx)\) associated with the closed form
\begin{equation*}
\mathfrak a(v,w)=\int_{\mathbb R^n}\nabla v\cdot\nabla\overline w\,dx,
\quad\mathcal D(\mathfrak a)=H^1(\mathbb R^n).
\end{equation*}
Since \(q\) is bounded above and below, the weighted \(L^2\)-norm is
equivalent to the usual one, so \(\mathfrak a\) is densely defined, closed,
and nonnegative. Hence the representation theorem defines \(A_q\).

We claim that $\mathcal D(A_q)=H^2(\mathbb R^n)$, $A_qv=-q^{-1}\Delta v$.
Indeed, \(v\in\mathcal D(A_q)\) if and only if some \(g\in L^2(\mathbb R^n)\)
satisfies
\begin{equation*}
\int_{\mathbb R^n}\nabla v\cdot\nabla\overline w\,dx
=\int_{\mathbb R^n}qg\overline w\,dx, \quad\forall\,w\in H^1(\mathbb R^n).
\end{equation*}
This means \(-\Delta v=qg\) distributionally. Since \(qg\in L^2(\mathbb R^n)\),
the Fourier multiplier estimate gives \(v\in H^2(\mathbb R^n)\). Conversely,
if \(v\in H^2(\mathbb R^n)\), then \(g=-q^{-1}\Delta v\in L^2(\mathbb R^n)\)
and the identity holds. The bounds on \(q\), together with
\(\|v\|_{H^2}\le C(\|\Delta v\|_{L^2}+\|v\|_{L^2})\), make the norm
equivalent to the \(H^2\)-norm.

For \(f\in L^2\), the spectral theorem gives the mild solution
\begin{equation*}
u(t)=\cos(tA_q^{1/2})f.
\end{equation*}
Moreover,
\begin{equation*}
U_\tau=\tau(\tau^2+A_q)^{-1}f\in H^2(\mathbb R^n),
\end{equation*}
which proves \eqref{eq:main_hyperbolic_laplace_equation}. Hence the Dirichlet
and weak Neumann traces in \eqref{eq:main_hyperbolic_laplace_data} lie in the
stated boundary spaces. This remains true even though the time-domain solution
is generated by a discontinuous \(L^2\) initial state. If classical
time-domain traces exist, their Laplace transforms agree with this definition.
\end{remark}

%%%%	

\begin{remark}\label{rem:why_laplace_data}
The use of \(U_\tau\), and of \(U_\lambda\) in the parabolic problem below,
is dictated by the regularity of the forward solution. Under the present
assumptions, a mild solution is only known to lie in
\(C([0,\infty);L^2(\mathbb R^n))\), which does not define the Dirichlet or
Neumann trace of \(u(t,\cdot)\) on \(\partial\Omega_\sharp\). In contrast, the
resolvent formulas place every Laplace transform in \(H^2(\mathbb R^n)\), so
its elliptic Cauchy traces are well defined. Thus the Laplace-domain family is
a rigorous representation of the passive observation at the regularity allowed
by the sharp cutoff.

If higher regularity is imposed, the measurement can be stated directly in the
time domain. For example, if the solution is integrable with values in
\(H^2(\mathbb R^n)\) after multiplication by the relevant exponential weight,
then the boundedness of the trace maps and the Bochner integral give
\begin{equation*}
\gamma_DU_\tau=\int_0^\infty e^{-\tau t}\gamma_Du(t)\,dt,
\quad \gamma_NU_\tau=\int_0^\infty e^{-\tau t}\gamma_Nu(t)\,dt.
\end{equation*}
The same holds with \(e^{-\lambda t}\) for the heat equation. Hence equality
of the full time-domain Cauchy traces implies equality of the Laplace-domain
data. This is an exact-data statement; it does not provide stability for
inversion of the Laplace transform.
\end{remark}

%%%

\begin{remark}\label{rem:dirichlet_only_observation}
Only the Dirichlet component of \eqref{eq:main_hyperbolic_laplace_data} is needed. For \(\tau>0\), \eqref{eq:main_hyperbolic_common_truncation} gives \(f=0\) and \(q=1\) in \(\mathbb R^n\setminus\overline{\Omega}_\sharp\), hence
\begin{equation*}
(-\Delta+\tau^2)U_\tau=0,\quad\text{in }\mathbb R^n\setminus\overline{\Omega}_\sharp.
\end{equation*}
Uniqueness follows from an \(H^1\) energy argument. Suppose
\(Z\in H^1(\mathbb R^n\setminus\overline{\Omega}_\sharp)\) solves
\((-\Delta+\tau^2)Z=0\) and has zero Dirichlet trace. Then
\(Z\in H_0^1(\mathbb R^n\setminus\overline{\Omega}_\sharp)\). Testing the weak
equation with \(\overline Z\) and taking real parts gives
\begin{equation*}
\int_{\mathbb R^n\setminus\overline{\Omega}_\sharp}
\bigl(|\nabla Z|^2+\tau^2|Z|^2\bigr)\,dx=0,
\end{equation*}
so \(Z=0\). The same holds on each connected component of the complement.
Thus the exterior \(H^1\) solution is determined by its Dirichlet trace, and
the exterior Dirichlet-to-Neumann map determines its weak Neumann trace. On
the inner boundary, the exterior unit normal is the negative of the unit
normal for \(\Omega_\sharp\), fixing \(\gamma_NU_\tau\) according to
\eqref{eq:main_hyperbolic_laplace_data}.

If the time-domain Dirichlet trace exists with the weighted integrability of
Remark~\ref{rem:why_laplace_data}, it determines \(\gamma_DU_\tau\) for every
\(\tau>0\), and the exterior argument then determines the full Laplace-domain
Cauchy data. At the present \(L^2\) regularity, the Laplace-domain Dirichlet
family is taken as the observation. We do not claim unconditional equivalence
with a time-domain trace or stability of inverse Laplace transformation; see
\cite{knox2020determining}.
\end{remark}

%%%

\begin{theorem}\label{thm:main_hyperbolic_quotient_uniqueness}
Let \(\Omega_\sharp\) be either an oblique cylinder
\(\Omega_{\theta,\mu,h}\) or a bounded Lipschitz polygonal cylinder
\(\Omega_{\mathcal B}\). Let \((f_\iota,c_\iota)\), \(\iota=1,2\), be two hyperbolically
admissible pairs relative to the same \(\Omega_\sharp\), and assume
\begin{equation*}
\Lambda_{f_1,c_1}^{\Omega_\sharp}(\tau)=\Lambda_{f_2,c_2}^{\Omega_\sharp}(\tau),
\quad\forall\,\tau>0.
\end{equation*}
Then
\begin{equation*}
\frac{f_1}{c_1^2}=\frac{f_2}{c_2^2},\quad\text{a.e. }\mathbb R^n.
\end{equation*}
\end{theorem}
The conclusion is equality of the compactly supported quotient \(qf\), not
separate recovery of \(f\) and \(c\). Separate recovery requires additional
assumptions; see, for instance, \cite{chen2025determining}.

%%%%%%%%%%%%%%%%%%%%%%%%%%%%%%%%%%%%%%%%%%%%%%%%%%

\subsection{Parabolic passive inverse problem}\label{subsec:main_parabolic_passive}

We consider
\begin{equation*}
\begin{cases}
q(x)\partial_tu(t,x)-\Delta u(t,x)=0,& (t,x)\in(0,\infty)\times\mathbb R^n,\\
u(0,x)=f(x),& x\in\mathbb R^n,
\end{cases}
\end{equation*}
where \(q=\kappa^{-1}\). For \(\lambda>0\), set
\begin{equation*}
U_\lambda(x):=\int_0^\infty e^{-\lambda t}u(t,x)\,dt.
\end{equation*}
Then
\begin{equation}\label{eq:main_parabolic_laplace_equation}
(-\Delta+\lambda q)U_\lambda=qf,\quad\text{in }\mathbb R^n,
\end{equation}
and the passive Cauchy observation is the family
\begin{equation}\label{eq:main_parabolic_laplace_data}
\mathcal P_{f,\kappa}^{\Omega_\sharp}(\lambda)
:=(\gamma_DU_\lambda,\gamma_NU_\lambda),\quad \lambda>0.
\end{equation}
This pair lies in
\(H^{1/2}(\partial\Omega_\sharp)\times
H^{-1/2}(\partial\Omega_\sharp)\), with the same weak Neumann trace
definition as in the hyperbolic problem.

\begin{definition}\label{def:parabolic_admissibility}
A pair \((f,\kappa)\) is \emph{parabolically admissible} relative to
\(\Omega_\sharp\) if the following hold.
\begin{enumerate}
\item \textbf{Sharp truncation and positivity.}
There exist \(f_{\mathrm{in}}\in L^2(\Omega_\sharp)\) and
\(\kappa_{\mathrm{in}}\in L^\infty(\Omega_\sharp)\) with
\(0<\kappa_-\le\kappa_{\mathrm{in}}\le\kappa_+<\infty\) such that
\begin{equation*}
f=f_{\mathrm{in}}\chi_{\Omega_\sharp},\quad \kappa=1+(\kappa_{\mathrm{in}}-1)\chi_{\Omega_\sharp}.
\end{equation*}
Thus \(f\) and \(\kappa-1\) are supported in
\(\overline{\Omega}_\sharp\) and may jump at
\(\partial\Omega_\sharp\). Equivalently, \(q\in L^\infty\)
satisfies \(0<q_-\le q\le q_+<\infty\) and
\(q=1\) outside \(\overline{\Omega}_\sharp\).

\item \textbf{Transverse quotient structure.}
If \(\Omega_\sharp=\Omega_{\theta,\mu,h}\), then
\begin{equation}\label{eq:main_parabolic_quotient_structure_translated}
\frac{f_{\mathrm{in}}(x)}{\kappa_{\mathrm{in}}(x)}=F(y),
\quad F\in L^1(D)\cap L^2(D), \quad \mbox{a.e. }\Omega_{\theta,\mu,h}.
\end{equation}
Hence $f/\kappa=F(y)\chi_{\Omega_{\theta,\mu,h}}$.

If \(\Omega_\sharp=\Omega_{\mathcal B}\), then there exist transverse
polynomials \(F_j\) such that
\begin{equation}\label{eq:main_parabolic_quotient_structure_broken}
\frac{f_{\mathrm{in}}(x)}{\kappa_{\mathrm{in}}(x)}
=\sum_{j=1}^N F_j(Y(x))\chi_{P_j}(x),
\quad\text{a.e. }\Omega_{\mathcal B}.
\end{equation}
Thus
\begin{equation*}
f/\kappa=\sum_{j=1}^N F_j(Y(x))\chi_{P_j}(x), \quad\text{in }\mathbb R^n.
\end{equation*}

\item \textbf{Mild solution, local boundedness and decay.}
Let \(u(t)=e^{-tA_q}f\) be the mild whole-space solution. Assume
\begin{equation*}
\sup_{t\ge0}\|u(t,\cdot)\|_{L^2(\Omega_\sharp)}<\infty,
\quad \|u(t,\cdot)\|_{L^2(\Omega_\sharp)}\to 0,
\quad\text{as }t\to\infty.
\end{equation*}
\end{enumerate}
\end{definition}

\begin{remark}\label{rem:parabolic_forward_wellposedness}
With \(A_q\) as above, the analytic semigroup \(e^{-tA_q}\) is well defined on
\(L^2(\mathbb R^n)\), and its resolvent gives
\begin{equation*}
U_\lambda=(\lambda+A_q)^{-1}f\in H^2(\mathbb R^n),
\end{equation*}
which is \eqref{eq:main_parabolic_laplace_equation}. Hence the Dirichlet and
weak Neumann traces in \eqref{eq:main_parabolic_laplace_data} lie in the stated
boundary spaces. This formulation uses the \(t>0\) smoothing while allowing a
sharply truncated initial state to lie only in \(L^2\).
\end{remark}

\begin{theorem}\label{thm:main_parabolic_quotient_uniqueness}
Let \(\Omega_\sharp\) be either an oblique cylinder or a polygonal cylinder, and let \((f_\iota,\kappa_\iota)\), \(\iota=1,2\), be parabolically admissible pairs relative to \(\Omega_\sharp\). If
\begin{equation*}
\mathcal P_{f_1,\kappa_1}^{\Omega_\sharp}(\lambda)
=\mathcal P_{f_2,\kappa_2}^{\Omega_\sharp}(\lambda),
\quad\forall\,\lambda>0,
\end{equation*}
then
\begin{equation*}
\frac{f_1}{\kappa_1}=\frac{f_2}{\kappa_2},\quad\text{a.e. }\mathbb R^n.
\end{equation*}
\end{theorem}

%%%%%%%%%%%%%%%%%%%%%%%%%%%%%%%%%%%%%%%%	

\subsection{Loss of uniqueness for a finite observation time}
\label{subsec:main_finite_time}

The preceding uniqueness theorems use the full time trace via all its Laplace
transforms. We now stop at a fixed time \(T>0\) and show that the missing tail
appears as a terminal-state residual.

Let
\begin{equation*}
\mathcal H(\Omega_\sharp):=\bigl\{v|_{\overline{\Omega}_\sharp}:
v\in C^\infty(U;\mathbb C),\ \Delta v=0 \text{ in } U,
\text{ for some open } U\supset\overline{\Omega}_\sharp\bigr\}.
\end{equation*}
Assume $u_\iota\in C([0,T];L^2(\Omega_\sharp))$, $\iota=1,2$. This does not define the time-domain Cauchy traces of \(u_\iota\). We therefore 
formulate finite-time data through a generalized Green functional. Set
\begin{equation*}
H_\Delta(\Omega_\sharp):=\{w\in L^2(\Omega_\sharp):\Delta w\in L^2(\Omega_\sharp)\},
\end{equation*}
with the norm
\begin{equation*}
\|w\|_{H_\Delta(\Omega_\sharp)}:=\left(\|w\|_{L^2(\Omega_\sharp)}^2
+\|\Delta w\|_{L^2(\Omega_\sharp)}^2\right)^{1/2}.
\end{equation*}
Denote by \(H_0^2(\Omega_\sharp)\) the closure of \(C_c^\infty(\Omega_\sharp)\) with respect to the \(H^2\)-norm, and set \(\mathfrak T(\Omega_\sharp):=H^2(\Omega_\sharp)/H_0^2(\Omega_\sharp)\).  Write \([\phi]=\phi+H_0^2(\Omega_\sharp)\), with quotient norm
\begin{equation*}
\|[\phi]\|_{\mathfrak T(\Omega_\sharp)}:=\inf_{\psi\in H_0^2(\Omega_\sharp)}
\|\phi+\psi\|_{H^2(\Omega_\sharp)}.
\end{equation*}
For \(w\in H_\Delta(\Omega_\sharp)\), define
\begin{equation}\label{eq:generalized_green_functional}
\mathcal G_w([\phi]):=\int_{\Omega_\sharp}
\bigl((\Delta w)\phi-w\Delta\phi\bigr)\,dx,\quad [\phi]\in\mathfrak T(\Omega_\sharp).
\end{equation}
This is well defined and continuous on \(\mathfrak T(\Omega_\sharp)\).
Indeed, for \(\psi\in H_0^2(\Omega_\sharp)\), approximation by
\(C_c^\infty(\Omega_\sharp)\) and the distributional definition of
\(\Delta w\) give
\begin{equation*}
\int_{\Omega_\sharp}(\Delta w)\psi\,dx=\int_{\Omega_\sharp}w\Delta\psi\,dx.
\end{equation*}
Thus \eqref{eq:generalized_green_functional} depends only on \([\phi]\).
Also, Cauchy--Schwarz and the fixed \(H^2\)-norm give
\begin{equation*}
|\mathcal G_w([\phi])|\le C_n\bigl(\|w\|_{L^2(\Omega_\sharp)}
+\|\Delta w\|_{L^2(\Omega_\sharp)}\bigr)\|\phi+\psi\|_{H^2(\Omega_\sharp)}
\end{equation*}
for all \(\psi\in H_0^2(\Omega_\sharp)\). The constant \(C_n\) depends only
on the fixed \(H^2\)-norm, not on \(w,\phi,\psi\).
Taking the infimum over all representatives of \([\phi]\) gives
\begin{equation*}
\|\mathcal G_w\|_{\mathfrak T(\Omega_\sharp)^*}
\le C_n\bigl(\|w\|_{L^2(\Omega_\sharp)}+\|\Delta w\|_{L^2(\Omega_\sharp)}\bigr).
\end{equation*}
So \(w\mapsto\mathcal G_w\) is continuous from
\(H_\Delta(\Omega_\sharp)\) to \(\mathfrak T(\Omega_\sharp)^*\). If \(w\)
and \(\phi\) are smooth up to the boundary, the second Green identity gives
\begin{equation*}
\mathcal G_w([\phi])=\int_{\partial\Omega_\sharp}
\bigl(\partial_\nu w\,\phi-w\partial_\nu\phi\bigr)\,dS.
\end{equation*}

Define
\begin{equation}\label{eq:WjT}
\eta(t)=
\begin{cases}
T-t, & \text{wave case},\\
1,   & \text{heat case},
\end{cases}
\quad
W_\iota^T(x):=\int_0^T \eta(t)u_\iota(t,x)\,dt.
\end{equation}
By Lemma~\ref{lem:finite_time_endpoint_identity},
\begin{equation*}
\Delta W_\iota^T=q_\iota\bigl(u_\iota(T)-f_\iota\bigr),
\end{equation*}
in both cases. Hence \(W_\iota^T,\Delta W_\iota^T\)\(\in L^2(\Omega_\sharp)\), and the
weighted generalized finite-time boundary datum is
\begin{equation*}
\mathcal M_\iota^T:=\mathcal G_{W_\iota^T}
\in\mathfrak T(\Omega_\sharp)^*.
\end{equation*}
Suppose \(a_T u_\iota \in L^1(0,T;H^2(\Omega_\sharp))\). Then
\(W_\iota^T\)\(\in H^2(\Omega_\sharp)\). Since the Dirichlet and weak Neumann trace
maps are bounded on \(H^2(\Omega_\sharp)\), they commute with the Bochner
integral. With the bilinear pairing $\langle\cdot,\cdot\rangle_{-1/2,1/2}:H^{-1/2}(\partial\Omega_\sharp)\times
H^{1/2}(\partial\Omega_\sharp)\to\mathbb C$ (no conjugation), the weak Green identity gives
\begin{equation*}
\mathcal M_\iota^T([\phi])=\int_0^T \eta(t)\Big(
\langle\gamma_Nu_\iota(t),\gamma_D\phi\rangle_{-1/2,1/2}
-\langle\gamma_N\phi,\gamma_Du_\iota(t)\rangle_{-1/2,1/2}\Big)\,dt.
\end{equation*}
Here  
\begin{equation*}
\gamma_Du_\iota(t),\gamma_D\phi\in H^{1/2}(\partial\Omega_\sharp),
\quad \gamma_Nu_\iota(t), \,\gamma_N\phi\in H^{-1/2}(\partial\Omega_\sharp).
\end{equation*}
If the boundary and the functions are smooth enough, \(\mathcal M_\iota^T([\phi])\) reduces to
\begin{equation*}
\int_0^T \eta(t)\int_{\partial\Omega_\sharp}
\bigl(\partial_\nu u_\iota(t)\,\phi-u_\iota(t)\,\partial_\nu\phi\bigr)\,dS\,dt.
\end{equation*}
Lemma~\ref{lem:finite_time_endpoint_identity} shows that \(\mathcal M_\iota^T\) is
the continuous extension, in \(\mathfrak T(\Omega_\sharp)^*\), of these
weighted boundary data; it is not a pointwise-in-time Cauchy trace.

For later reference, define the boundary-functional difference on harmonic test functions by
\begin{equation*}
\mathfrak D_T(v):=(\mathcal M_1^T-\mathcal M_2^T)([v]),
\quad v\in\mathcal H(\Omega_\sharp).
\end{equation*}
Set \(q_\iota=c_\iota^{-2}\) or \(q_\iota=\kappa_\iota^{-1}\), and
\begin{equation*}
G=q_1f_1-q_2f_2,\quad R_T=q_1u_1(T,\cdot)-q_2u_2(T,\cdot),
\quad \varepsilon_T=\|R_T\|_{L^2(\Omega_\sharp)}.
\end{equation*}

%%%

\begin{theorem}	\label{thm:main_finite_time_stability}
Let the two coefficient--initial-state pairs satisfy the support, positivity,
and structural conditions in the corresponding admissibility definition, and
assume
\begin{equation*}
\mathcal M_1^T=\mathcal M_2^T,\quad\text{in }\mathfrak T(\Omega_\sharp)^*.
\end{equation*}
\begin{item}
\item[(1) Oblique cylinder.]
Let \(\Omega_\sharp=\Omega_{\theta,\mu,h}\) and
\begin{equation*}
G(x)=G_0(y)\chi_{\Omega_{\theta,\mu,h}}(x),\quad \widetilde G_0:=G_0\chi_D\in H^S(\mathbb R^{n-1}),\quad\|\widetilde G_0\|_{H^S}\le M,
\end{equation*}
with \(S>0\) and \(M>0\). Then for every \(0\le r<S\),
\begin{equation}\label{eq:main_finite_time_holder}
\|\widetilde G_0\|_{H^r(\mathbb R^{n-1})}
\le C M^{1-\alpha_r}\varepsilon_T^{\alpha_r},
\quad \alpha_r=\frac{S-r}{S+n/2}. 
\end{equation}
Here \(C=C(n,r,S,D,\mu,h)\) is independent of \(G_0,M,T\). For \(r=0\),
\begin{equation*}
\|q_1f_1-q_2f_2\|_{L^2(\Omega_{\theta,\mu,h})}
=h^{1/2}\|G_0\|_{L^2(D)}\le C M^{1-\alpha_0}\varepsilon_T^{\alpha_0},
\quad \alpha_0=\frac{S}{S+n/2},
\end{equation*}
with \(h^{1/2}\) absorbed into \(C\).

\item[(2) Polygonal cylinder.]
Let \(\Omega_\sharp=\Omega_{\mathcal B}\), fix \(d\in\mathbb N_0\), and let
\begin{equation*}
G(x)=\sum_{j=1}^N p_j(Y(x))\chi_{P_j}(x),\quad p_j\in\Pi_d(D),
\end{equation*}
where \(\Pi_d(D)\) is the space of polynomials of total degree at most \(d\)
restricted to \(D\). Set
\begin{equation*}
M_d=\binom{n-1+d}{d},\quad q_d=NM_d.
\end{equation*}
Let \(\Xi=\{(\rho_k,\omega_k):1\le k\le q_d\}\) be any nondegenerate sampling
set from Theorem~\ref{thm:polygonal_finite_dimensional_stability}. Then
\begin{equation}\label{eq:main_finite_time_polygonal}
\|G\|_{L^2(\Omega_{\mathcal B})} \le CK_{\mathcal B,\Xi}\varepsilon_T,
\end{equation}
where
\begin{equation*}
K_{\mathcal B,\Xi}=\left(\sum_{k=1}^{q_d}\|e^{-i(\cdot-x_0)\cdot
\zeta_\theta(\rho_k\omega_k)}\|_{L^2(\Omega_{\mathcal B})}^2
\right)^{1/2}.
\end{equation*}
\end{item}
\end{theorem}

Section~\ref{sec:finite_time_stability} proves the theorem. Its endpoint identity
also handles unequal finite-time measurements via the boundary functional
\(\mathfrak D_T\). For equal measurements, the only remaining defect is the
terminal residual; hence \(R_T=0\) restores exact uniqueness, and a known decay
bound for \(R_T\) gives a quantitative approximation to the infinite-time
result. If the zero extension of \(G_0\) jumps across \(\partial D\), the
natural range is \(0<S<1/2\). For \(S\ge1/2\), one needs the corresponding trace
compatibility at \(\partial D\).

%%%%%%%%%%%%%%%%%%%%%%%%%%%%%%%%%%%%%%%%%%%%%%%%%%%%%%%%%%%%%%%%%%%%%	

\section{A Totally Real Uniqueness Argument via a Weighted Integral Identity}
\label{sec:totally_real_core}

We prove Theorem~\ref{thm:main_totally_real_uniqueness} in the oblique cylinder geometry of Subsection~\ref{subsec:Oblique_cylinder_geometry}.
Fix \(\theta\in\mathbb S^{n-1}\) and write
\begin{equation*}
x=x_0+y+s\theta,
\qquad y\in\theta^\perp,\quad s\in\mathbb R.
\end{equation*}
Identifying \(\theta^\perp\) via an orthonormal basis and
using the Euclidean measure \(dy\), the map \((y,s)\mapsto x_0+y+s\theta\) has
Jacobian one. Recall the domain \(\Omega_{\theta,\mu,E}\), and the unknown is
codimension-one:
\begin{equation}\label{eq:g_codimension_one_structure_TR}
g(x)=g_0(y)\chi_{\Omega_{\theta,\mu,E}}(x),	\quad g_0\in L^1(D).
\end{equation}
We prove the theorem for the phase \(e^{-ix\cdot\zeta_\theta(\eta)}\); the
opposite sign follows by replacing the parametrization with its negative.

%%%%%%%%%%%%%%%%%%%%%%%%%%%%%%%%%%%%%%%%%

\subsection{Reduction to an entire Fourier--Laplace transform}
\label{subsec:FL_reduction_TR}

Assume
\begin{equation}\label{eq:basic_weighted_identity}
\int_{\Omega_{\theta,\mu,E}}g(x)e^{-ix\cdot\zeta_\theta(\eta)}\,dx=0,
\end{equation}
where $\zeta_\theta(\eta)=\eta+i|\eta|\theta$ and $\eta\in U\Subset\theta^\perp\setminus\{0\}$. For \(x=x_0+y+s\theta\), using \(\eta\cdot\theta=y\cdot\theta=0\) and \(|\theta|=1\),
\begin{equation*}
x\cdot\zeta_\theta(\eta)=x_0\cdot\zeta_\theta(\eta)+y\cdot\eta+i|\eta|s,
\end{equation*}
hence
\begin{equation}\label{eq:phase_factorization_general}
e^{-ix\cdot\zeta_\theta(\eta)}=e^{-ix_0\cdot\zeta_\theta(\eta)}
e^{-i\eta\cdot y}e^{|\eta|s}.
\end{equation}
Since \(D\) and \(E\) are bounded and \(g_0\in L^1(D)\), the integrand is
integrable, so Fubini and \eqref{eq:g_codimension_one_structure_TR} give
\begin{equation}\label{eq:Fubini_TR_general}
0=e^{-ix_0\cdot\zeta_\theta(\eta)}\int_D g_0(y)e^{-i\eta\cdot y}
\left(\int_{\mu\cdot y+E}e^{|\eta|s}\,ds\right)dy,
\end{equation}
where \(\mu\cdot y+E:=\{\mu\cdot y+r:r\in E\}\).

Set \(\rho:=|\eta|>0\). The change of variables \(s=\mu\cdot y+r\) gives
\begin{equation}\label{eq:section_laplace_factorization}
\int_{\mu\cdot y+E}e^{\rho s}\,ds=e^{\rho\mu\cdot y}\int_E e^{\rho r}\,dr.
\end{equation}
Define the one-dimensional Laplace factor as
\begin{equation}
B_E^+(\rho):=\int_E e^{\rho r}\,dr,\quad \rho>0.
\end{equation}
Since \(E\) has positive measure, \(B_E^+(\rho)>0\), and boundedness of \(E\)
gives \(B_E^+(\rho)<\infty\). Hence the nonzero scalar factors in
\eqref{eq:Fubini_TR_general} can be divided out, yielding
\begin{equation}\label{eq:reduced_integral_identity_TR}
\int_D g_0(y)\exp\bigl[(|\eta|\mu-i\eta)\cdot y\bigr]\,dy=0,\quad \eta\in U.
\end{equation}
This is where the translated section structure of the oblique cylinder geometry enters: the one-dimensional sections must be translates of a common set \(E\); otherwise
\eqref{eq:section_laplace_factorization} would not separate into a scalar
Laplace factor and an exponential linear in \(y\).

Define the Fourier--Laplace transform of \(g_0\) by
\begin{equation}\label{eq:H_def_TR}
H(z):=\int_D g_0(y)e^{z\cdot y}\,dy,\quad z\in\mathbb C^m,
\end{equation}
where \(z\cdot y=\sum_{j=1}^m z_jy_j\) under the fixed identification
\(\theta^\perp\cong\mathbb R^m\). Then
\eqref{eq:reduced_integral_identity_TR} becomes
\begin{equation}\label{eq:H_vanish_on_surface_TR}
H(|\eta|\mu-i\eta)=0,\quad \eta\in U.
\end{equation}
Since \(D\) is bounded and \(g_0\in L^1(D)\), differentiation under the
integral sign shows that \(H\) is holomorphic on \(\mathbb C^m\).

\begin{lemma}\label{lem:H_entire_TR}
The function \(H\) in \eqref{eq:H_def_TR} is entire on \(\mathbb C^m\).
Moreover, for every multi-index \(\gamma\in\mathbb N_0^m\),
\begin{equation}\label{eq:H_derivative_formula_TR}
\partial_z^\gamma H(z)=\int_D y^\gamma g_0(y)e^{z\cdot y}\,dy,
\quad z\in\mathbb C^m,
\end{equation}
where \(y^\gamma=y_1^{\gamma_1}\cdots y_m^{\gamma_m}\).
\end{lemma}
\begin{proof}
Fix a compact set \(K\Subset\mathbb C^m\). Since \(D\) is bounded, choose
\(R_D>0\) with \(|y|\le R_D\) for \(y\in D\), and set
\(M_K:=\sup_{z\in K}|z|<\infty\). For \(z\in K\), \(y\in D\),
\begin{equation*}
|e^{z\cdot y}|=e^{\operatorname{Re}(z\cdot y)}\le e^{|z||y|}\le e^{M_KR_D}.
\end{equation*}
Hence \(|g_0(y)e^{z\cdot y}|\le e^{M_KR_D}|g_0(y)|\), which is integrable on
\(D\). Thus \(H\) is well defined and locally bounded.

To justify differentiation under the integral sign, fix \(z\in\mathbb C^m\)
and a coordinate vector \(e_j\). For small \(h\),
\begin{equation*}
\frac{e^{(z+he_j)\cdot y}-e^{z\cdot y}}{h}=y_j e^{z\cdot y}\int_0^1 e^{thy_j}\,dt.
\end{equation*}
For \(z,z+he_j\in K\) and \(|h|\le1\), this is bounded by $R_De^{(M_K+1)R_D}|g_0(y)|$,
which is integrable. Dominated convergence gives
\begin{equation*}
\partial_{z_j}H(z)=\int_D y_jg_0(y)e^{z\cdot y}\,dy.
\end{equation*}
Repeating the argument yields
\eqref{eq:H_derivative_formula_TR}.

Finally, for fixed \(z_0\in\mathbb C^m\),
\begin{equation*}
e^{z\cdot y}=e^{z_0\cdot y}\sum_{\gamma\in\mathbb N_0^m}
\frac{y^\gamma(z-z_0)^\gamma}{\gamma!}.
\end{equation*}
On a small ball about \(z_0\), the series is dominated by \(C|g_0(y)|\).
Termwise integration gives a convergent power series with coefficients
\(\partial_z^\gamma H(z_0)/\gamma!\). Hence \(H\) is holomorphic at every
point of \(\mathbb C^m\), and therefore entire.
\end{proof}

%%%%%%%%%%%%%%%%%%%%%%%%%%%%%

\subsection{The associated real-analytic totally real submanifold}

The vanishing relation \eqref{eq:H_vanish_on_surface_TR} is parametrized by
\begin{equation}\label{eq:Phi_def_TR}
\Phi:U\to\mathbb C^{n-1},\quad \Phi(\eta):=|\eta|\mu-i\eta,
\end{equation}
and we set
\begin{equation}\label{eq:MU_def_TR}
M_U:=\Phi(U)\subset\mathbb C^{n-1}.
\end{equation}
Thus \eqref{eq:H_vanish_on_surface_TR} states that \(H\) vanishes on \(M_U\).
We next show that \(M_U\) is a maximally totally real, \(n-1\)-dimensional
real-analytic submanifold of \(\mathbb C^{n-1}\).
\begin{definition}
Let \(M\subset\mathbb C^{n-1}\) be an \(n-1\)-dimensional \(C^1\) real submanifold.
It is \emph{totally real at \(p\in M\)} if
\begin{equation*}
T_pM\cap iT_pM=\{0\}.
\end{equation*}
If this holds at every \(p\in M\), then \(M\) is \emph{totally real}. Since
\(\dim_{\mathbb R}M=n-1=\dim_{\mathbb C}\mathbb C^{n-1}\), such an \(M\) is also
called maximally totally real.
\end{definition}

The next lemma is where the special phase \(\eta+i|\eta|\theta\) and the affine
translation by \(\mu\) enter the argument. The proof is included in full because
this is the point at which the lower-dimensional family of weights becomes an
injective testing family after the structural reduction.

\begin{lemma}\label{lem:MU_totally_real_TR}
The set \(M_U\) defined in \eqref{eq:MU_def_TR} is an embedded real-analytic,
maximally totally real submanifold of \(\mathbb C^{n-1}\).
\end{lemma}

\begin{proof}
We show that \(\Phi\) is a real-analytic embedding and that \(M_U\) is totally real.

Since \(U\Subset\mathbb R^{n-1}\setminus\{0\}\), the map \(\eta\mapsto|\eta|\) is
real analytic, so \(\Phi\) is real analytic. Moreover,
\begin{equation*}
\operatorname{Re}\Phi(\eta)=|\eta|\mu,
\quad \operatorname{Im}\Phi(\eta)=-\eta.
\end{equation*}
Thus \(\eta=-\operatorname{Im}z\) for \(z\in M_U\), so \(\Phi\) is injective
with real-analytic inverse \(z\mapsto-\operatorname{Im}z\).

For \(a\in\mathbb R^{n-1}\),
\begin{equation*}
d\Phi_\eta(a)=\mu\frac{\eta\cdot a}{|\eta|}-ia.
\end{equation*}
If \(d\Phi_\eta(a)=0\), then taking imaginary parts gives \(a=0\). Hence
\(\Phi\) is an immersion, and therefore an embedding. Thus \(M_U\) is an
\(m\)-dimensional real-analytic submanifold of \(\mathbb C^{n-1}\).

It remains to prove total reality. Fix \(\eta\in U\), set
\(e=\eta/|\eta|\), and suppose
\begin{equation*}
v\in T_{\Phi(\eta)}M_U\cap iT_{\Phi(\eta)}M_U.
\end{equation*}
Then for some \(a,b\in\mathbb R^{n-1}\),
\begin{equation*}
v=\mu(e\cdot a)-ia,\quad v=i\{\mu(e\cdot b)-ib\}=b+i\mu(e\cdot b).
\end{equation*}
Comparing real and imaginary parts,
\begin{equation*}
b=\mu(e\cdot a),\quad -a=\mu(e\cdot b).
\end{equation*}
Substituting the first identity into the second gives
\begin{equation*}
a=-\mu(e\cdot\mu)(e\cdot a).
\end{equation*}
Taking the scalar product with \(e\),
\begin{equation*}
\bigl(1+(e\cdot\mu)^2\bigr)(e\cdot a)=0.
\end{equation*}
Since the coefficient is positive, \(e\cdot a=0\), hence \(b=0\) and
\(a=0\). Therefore \(v=0\), so
\begin{equation*}
T_{\Phi(\eta)}M_U\cap iT_{\Phi(\eta)}M_U=\{0\}.
\end{equation*}
Thus \(M_U\) is maximally totally real.
\end{proof}

%%%%%%%%%%%%%%%%%%%%%%%%%%%%%%%%%%%%%%%%%%%%%%%%%%%%%%

\subsection{Holomorphic uniqueness from a totally real submanifold}

We now prove the holomorphic uniqueness principle used below. Although standard,
we include the argument to make the role of real analyticity explicit.

\begin{lemma}\label{lem:local_graph_TR}
Let \(M\subset\mathbb C^{n-1}\) be a real-analytic maximally totally real
submanifold and \(p\in M\). After a translation and a complex-linear change
of coordinates, \(M\) is locally the graph
\begin{equation}\label{eq:local_graph_form}
M=\{x+i\varphi(x):x\in V\subset\mathbb R^m\},
\end{equation}
where \(V\) is a neighborhood of \(0\) and \(\varphi:V\to\mathbb R^{n-1}\) is
real analytic with \(\varphi(0)=0\) and \(D\varphi(0)=0\).
\end{lemma}
\begin{proof}
After translating, assume \(p=0\), and set \(P:=T_0M\). Since \(M\) is
maximally totally real at \(0\),
\begin{equation*}
P\cap iP=\{0\},
\end{equation*}
so \(\mathbb C^{n-1}=P\oplus iP\) as real vector spaces. Choose a real basis
\(v_1,\dots,v_{n-1}\) of \(P\). These vectors are complex-linearly independent:
if \(\sum_j(a_j+ib_j)v_j=0\) with \(a_j,b_j\in\mathbb R\), then
\(\sum_j a_jv_j=-i\sum_j b_jv_j\in P\cap iP\), so both sums vanish and hence
\(a_j=b_j=0\). Thus an invertible complex-linear map sends \(P\) to
\(\mathbb R^{n-1}\). After applying it, assume $T_0M=\mathbb R^{n-1}\subset\mathbb C^{n-1}$.

Let \(\gamma:W\subset\mathbb R^{n-1}\to\mathbb C^{n-1}\) be a real-analytic local
parametrization of \(M\) with \(\gamma(0)=0\) and
\(d\gamma_0(\mathbb R^{n-1})=\mathbb R^{n-1}\). Write
\begin{equation*}
\gamma(t)=\operatorname{Re}\gamma(t)+i\operatorname{Im}\gamma(t).
\end{equation*}
Since \(d\gamma_0\) maps onto \(\mathbb R^{n-1}\) and the imaginary part of
\(T_0M\) is zero, the differential of \(t\mapsto\operatorname{Re}\gamma(t)\)
at \(0\) is invertible. By the real-analytic inverse function theorem,
\begin{equation*}
t\mapsto x:=\operatorname{Re}\gamma(t),
\end{equation*}
is a real-analytic diffeomorphism from a neighborhood of \(0\) onto some
neighborhood \(V\subset\mathbb R^{n-1}\). Reparametrizing \(M\) by \(x\),
\begin{equation*}
\gamma(t(x))=x+i\varphi(x),\quad \varphi(x):=\operatorname{Im}\gamma(t(x)).
\end{equation*}
Then \(\varphi\) is real analytic, \(\varphi(0)=0\), and
\(D\varphi(0)=0\) because \(T_0M=\mathbb R^{n-1}\). This gives
\eqref{eq:local_graph_form}.
\end{proof}

The normal form reduces uniqueness to the familiar fact that a holomorphic
function vanishing on a real product set vanishes in a complex neighborhood. We
state this in the next lemma, which bridges the real geometry and the
complex-analytic conclusion.
\begin{lemma}\label{lem:holomorphic_uniqueness_TR}
Let \(O\subset\mathbb C^{n-1}\) be connected and open, and let \(F\) be
holomorphic in \(O\). If \(F\) vanishes on a nonempty relatively open subset
of a real-analytic maximally totally real submanifold \(M\subset O\), then
\(F\equiv0\) in \(O\).
\end{lemma}	
\begin{proof}
Choose \(p\in M\) such that \(F\) vanishes on a relatively open neighborhood of
\(p\) in \(M\). By Lemma~\ref{lem:local_graph_TR}, after a translation and a
complex-linear change of coordinates, assume \(p=0\) and, near \(0\),
\begin{equation*}
M=\{x+i\varphi(x):x\in V\subset\mathbb R^{n-1}\},
\end{equation*}
with \(\varphi\) real analytic, \(\varphi(0)=0\), \(D\varphi(0)=0\), and
\begin{equation*}
F(x+i\varphi(x))=0,\quad x\in V.
\end{equation*}
Since \(\varphi\) is real analytic, it has a holomorphic extension
\(\varphi_{\mathbb C}\) to a complex neighborhood \(V_{\mathbb C}\) of \(0\).
Define
\begin{equation*}
\Psi(z):=z+i\varphi_{\mathbb C}(z).
\end{equation*}
Then \(D\Psi(0)=I\), so by the holomorphic inverse function theorem \(\Psi\) is
biholomorphic from a neighborhood of \(0\) onto a neighborhood of \(0\), after
shrinking \(V_{\mathbb C}\). Set
\begin{equation*}
G(z):=F(\Psi(z)).
\end{equation*}
Then \(G\) is holomorphic near \(0\), and \(G(x)=0\) for real \(x\in V\).

Choose a product box \(I_1\times\cdots\times I_{n-1}\Subset V\) of open intervals
containing \(0\). We show \(G\equiv0\) in a polydisc around \(0\). Fix real
\((x_2,\cdots,x_{n-1})\in I_2\times\cdots\times I_{n-1}\). The holomorphic function
\(z_1\mapsto G(z_1,x_2,\cdots,x_{n-1})\) vanishes on \(I_1\), so it vanishes
identically in a complex disk about \(0\). Thus \(G(z_1,x_2,\cdots,x_{n-1})=0\) for
complex \(z_1\) near \(0\) and real \((x_2,\cdots,x_{n-1})\) in the box. Fixing such
\(z_1\) and real \((x_3,\cdots,x_{n-1})\), the same one-variable argument applied to
\(z_2\mapsto G(z_1,z_2,x_3,\cdots,x_{n-1})\) extends the vanishing in \(z_2\).
Repeating successively for \(z_3,\cdots,z_{n-1}\), we obtain \(G\equiv0\) in a
polydisc around \(0\).

Since \(\Psi\) is biholomorphic, \(F\equiv0\) in a neighborhood of \(p\). By the
identity theorem for holomorphic functions of several variables, the zero set
of \(F\) is all of the connected domain \(O\). Hence \(F\equiv0\) in \(O\).
\end{proof}	

%%%%%%%%%%%%%%%%%%%%%%%%%%%%%%%%%%%%%%%%%%%%%%%%%%%%%%%%%%%%%%%%%%%%%

\subsection{Proof of the totally real uniqueness theorem}

\begin{proof}[Proof of Theorem~\ref{thm:main_totally_real_uniqueness}]
Assume \eqref{eq:main_TR_identity_minus}. By the reduction in
Subsection~\ref{subsec:FL_reduction_TR}, the Fourier--Laplace transform \(H\)
from \eqref{eq:H_def_TR} satisfies
\begin{equation*}
H|_{M_U}=0,\quad M_U=\{|\eta|\mu-i\eta:\eta\in U\}.
\end{equation*}
By Lemma~\ref{lem:H_entire_TR}, \(H\) is entire on \(\mathbb C^{n-1}\), and by
Lemma~\ref{lem:MU_totally_real_TR}, \(M_U\) is a real-analytic maximally
totally real submanifold of \(\mathbb C^{n-1}\). Since \(U\) is open and nonempty,
\(M_U\) is a nonempty relatively open subset of this submanifold. Hence
Lemma~\ref{lem:holomorphic_uniqueness_TR}, with \(O=\mathbb C^{n-1}\), gives
\begin{equation}
H\equiv0, \quad\text{in }\mathbb C^{n-1}.
\label{eq:H_identically_zero}
\end{equation}

Evaluating on the imaginary axis, for every \(\lambda\in\mathbb R^{n-1}\),
\begin{equation*}
0=H(-i\lambda)=\int_D g_0(y)e^{-i\lambda\cdot y}\,dy.
\end{equation*}
Let \(\widetilde g_0\) be the zero extension of \(g_0\) to \(\mathbb R^{n-1}\).
Since \(g_0\in L^1(D)\) and \(D\) is bounded, \(\widetilde g_0\in L^1(\mathbb R^{n-1})\).
Thus
\begin{equation*}
\widehat{\widetilde g_0}(\lambda)=\int_{\mathbb R^m}\widetilde g_0(y)e^{-i\lambda\cdot y}\,dy=0, \quad \lambda\in\mathbb R^{n-1}.
\end{equation*}
By injectivity of the Fourier transform on \(L^1(\mathbb R^{n-1})\),
\(\widetilde g_0=0\) a.e.; hence \(g_0=0\) a.e. in \(D\). By the structural
assumption \eqref{eq:g_codimension_one_structure_TR},
\begin{equation*}
g(x)=g_0(y)\chi_{\Omega_{\theta,\mu,E}}(x)=0, \quad  \mbox{a.e.~}\Omega_{\theta,\mu,E}.
\end{equation*}
For the sign-reversed identity, assume
\begin{equation*}
\int_{\Omega_{\theta,\mu,E}}g(x)e^{ix\cdot\zeta_\theta(\eta)}\,dx=0,
\quad \eta\in U.
\end{equation*}
Using (\ref{eq:phase_factorization_general}), Fubini's theorem gives
\begin{equation*}
B_E^-(|\eta|)
\int_D g_0(y)\exp\bigl[(-|\eta|\mu+i\eta)\cdot y\bigr]\,dy=0,
\end{equation*}
where $B_E^-(\rho):=\int_E e^{-\rho r}\,dr$. Since \(B_E^-(\rho)>0\) for \(\rho>0\),
$H(i\eta-|\eta|\mu)=0$, $\eta\in U$. But
\begin{equation*}
\Phi_+(\eta):=i\eta-|\eta|\mu=-\Phi(\eta),
\end{equation*}
and multiplication by \(-1\) preserves embeddedness and the maximally totally
real property. Hence \(\Phi_+(U)\) is again a real-analytic maximally totally
real submanifold. The same holomorphic uniqueness and Fourier transform
arguments yield \(H\equiv0\) and \(g=0\) a.e. This proves both assertions.
\end{proof}

%%%%%%%%%%%%%%%%%%%%%%%%%%%%%%%%%%%%%%%%%%%%%%%%%%%%%%%%%%%%%%%%%%%%%%

\section{A Brion--Barvinok Injectivity Argument for Polygonal Cylinders}
\label{sec:brion_barvinok_core}

We prove Theorem~\ref{thm:main_brion_barvinok_uniqueness} for the negative
phase; the positive-phase case follows by reversing the endpoint order. The
argument combines the weighted Brion--Barvinok expansion with separation of the
longitudinal endpoints and the vertex phases.

Throughout this section, the pieces \(P_j\), the polygonal cylinder \(\Omega_{\mathcal B}\), and the endpoint phases \(L_j(v;\omega),U_j(v;\omega)\) are as defined in \eqref{eq:polygonal_cumulative_data}--\eqref{eq:lower_upper_height_main}.

%%%%%%%%%%%%%%%%%%%%%%%%%%%%%%%%%%%%%%%%
\subsection{Reduction of the polygonal cylinder identity}

Assume first that the identity with phase \(e^{-ix\cdot\zeta_\theta(\eta)}\) holds. Let
\begin{equation*}
G(x)=\sum_{j=1}^N p_j(Y(x))\chi_{P_j}(x),\quad p_j\in\mathbb C[y_1,\ldots,y_{n-1}].
\end{equation*}
On the \(j\)-th piece, write \(t=t_{j-1}+\tau\), \(0<\tau<h_j\). Then
\begin{equation*}
\Phi(y,t)=x_0+y+r_{j-1}+\tau(\gamma_j+\theta).
\end{equation*}
The Jacobian of \((y,\tau)\mapsto\Phi(y,t_{j-1}+\tau)\) is \(1\). Put \(\rho=|\eta|\) and $\lambda_j(\eta):=\rho-i\gamma_j\cdot\eta$. Dropping the common nonzero factor \(e^{-ix_0\cdot\zeta_\theta(\eta)}\), Fubini's theorem gives
\begin{align*}
0&=\sum_{j=1}^N
e^{-ir_{j-1}\cdot\zeta_\theta(\eta)}\int_D p_j(y)e^{-i\eta\cdot y}\,dy
\int_0^{h_j}e^{\tau\lambda_j(\eta)}\,d\tau \notag\\
&=\sum_{j=1}^N\frac{e^{-ir_j\cdot\zeta_\theta(\eta)}
-e^{-ir_{j-1}\cdot\zeta_\theta(\eta)}}{\lambda_j(\eta)}
\int_D p_j(y)e^{-i\eta\cdot y}\,dy.
\end{align*}
Define the transverse Fourier--Laplace transforms
\begin{equation*}
H_j(z):=\int_Dp_j(y)e^{z\cdot y}\,dy, \quad z\in\mathbb C^m.
\end{equation*}
Then
\begin{equation}\label{eq:BB_reduced_minus}
\sum_{j=1}^N\frac{e^{-ir_j\cdot\zeta_\theta(\eta)}
-e^{-ir_{j-1}\cdot\zeta_\theta(\eta)}}{\lambda_j(\eta)}H_j(-i\eta)=0,
\quad \rho=|\eta|.
\end{equation}

For the opposite sign, after removing the common factor
\(e^{ix_0\cdot\zeta_\theta(\eta)}\), the same computation gives
\begin{equation}\label{eq:BB_reduced_plus}
\sum_{j=1}^N\frac{e^{ir_{j-1}\cdot\zeta_\theta(\eta)}
-e^{ir_j\cdot\zeta_\theta(\eta)}}{\lambda_j(\eta)}H_j(i\eta)=0,\quad \rho=|\eta|,
\end{equation}
where \(\operatorname{Re}\lambda_j(\eta)=\rho>0\), so no denominator vanishes
for \(\eta\ne0\). Along \(\eta=\rho\omega\), the endpoint exponentials in these
identities are governed by \(U_j(v;\omega)\) and \(L_j(v;\omega)\) via the
vertex expansion of \(H_j\).

The reduced expressions in \eqref{eq:BB_reduced_minus} and
\eqref{eq:BB_reduced_plus} are real analytic on
\(\theta^\perp\setminus\{0\}\). Let \(\mathcal C_U\) be a connected component
of \(\theta^\perp\setminus\{0\}\) such that $U\cap\mathcal C_U\neq\varnothing$.
Such a component exists since \(U\) is nonempty. If either reduced expression
vanishes on \(U\), then it vanishes on the nonempty open subset
\(U\cap\mathcal C_U\) of \(\mathcal C_U\), hence on all of \(\mathcal C_U\) by
real-analytic continuation. Set $\mathcal O_U:=\mathcal C_U\cap\mathbb S^{n-2}$.
Thus, for every \(\omega\in\mathcal O_U\), the ray $\eta=\rho\omega$, $\rho>0$,
is admissible for the large-frequency separation argument.

%%%%%%%%%%%%%%%%%%%%%%%%%%%%%%%%%%%%%%%%%%%%%%%%%%%%%%

\subsection{Weighted Brion--Barvinok vertex expansion}

We next recall the precise weighted vertex expansion needed below. The
unweighted ingredient is Barvinok's Theorem~2.6 \cite{barvinok1993computing}, which gives the vertex-cone formula for exponential integrals. In Step~3 below, we derive the weighted identity from that meromorphic formula by applying the finite-order differential operator \(p(\partial_z)\). We spell out the simple-polytope form because the explicit rational coefficients are used later. For a modern Fourier-analytic treatment of polytopal transforms and their vertex structure, see also \cite{machado2023null}.

\begin{lemma}\label{lem:weighted_BB_vertex_expansion}
Let \(D\subset\mathbb R^{n-1}\) be a compact simple convex polytope with nonempty
interior, and let \(p\in\mathbb C[y_1,\ldots,y_{n-1}]\). For each vertex
\(v\in V(D)\), let \(u_{v,1},\ldots,u_{v,n-1}\) be the inward edge directions
generating the tangent cone
\begin{equation*}
C_v=\left\{v+\sum_{k=1}^{n-1} s_ku_{v,k}:s_k\ge0\right\}.
\end{equation*}
Then there exist rational functions \(R_{D,p,v}(z)\), \(v\in V(D)\), such that
\begin{equation}\label{eq:weighted_BB_expansion}
\int_D p(y)e^{z\cdot y}\,dy=\sum_{v\in V(D)}e^{z\cdot v}R_{D,p,v}(z),
\end{equation}
as a meromorphic identity in \(z\in\mathbb C^{n-1}\). If
\begin{equation*}
p\left(v+\sum_{k=1}^{n-1} s_ku_{v,k}\right)
=\sum_{\alpha\in\mathbb N_0^{n-1}}c_\alpha^{(v)}s^\alpha,
\end{equation*}
then
\begin{equation}\label{eq:explicit_R_Dpv}
R_{D,p,v}(z)=|\det(u_{v,1},\ldots,u_{v,n-1})|
\sum_\alpha c_\alpha^{(v)}\frac{(-1)^{|\alpha|+n-1}\alpha!}
{\prod_{k=1}^{n-1}(z\cdot u_{v,k})^{\alpha_k+1}}.
\end{equation}
Moreover, if \(p\not\equiv0\), then
\begin{equation*}
R_{D,p,v}\not\equiv0, \quad \forall \,v\in V(D).
\end{equation*}
\end{lemma}

\begin{proof}
\noindent {Step 1.} Set
\begin{equation*}
H_{D,p}(z):=\int_D p(y)e^{z\cdot y}\,dy,\quad z\in\mathbb C^{n-1}.
\end{equation*}
Since \(D\) is compact and \(p\) is polynomial, the integral converges
absolutely for every \(z\in\mathbb C^{n-1}\). Differentiation under the integral
sign gives
\begin{equation*}
\partial_z^\alpha H_{D,p}(z)=\int_D y^\alpha p(y)e^{z\cdot y}\,dy.
\end{equation*}
Hence \(H_{D,p}\) is entire. The Brion--Barvinok expansion represents it as a
sum of meromorphic vertex terms whose poles cancel in the total sum.

\noindent {Step 2.}
Fix a vertex \(v\in V(D)\). Since \(D\) is simple, the tangent cone at \(v\) is generated by \(m\) linearly independent inward edge directions \(u_{v,1},\ldots,u_{v,n-1}\):
\begin{equation*}
C_v=\left\{v+\sum_{k=1}^{n-1} s_k u_{v,k}:s_k\ge0\right\}.
\end{equation*}
In the region $\operatorname{Re}(z\cdot u_{v,k})<0$, $k=1,\cdots,n-1$, the cone integral \(\int_{C_v}p(y)e^{z\cdot y}\,dy\) converges absolutely. Using $
y=v+\sum_{k=1}^{n-1} s_k u_{v,k}$, $s_k\ge0$, we get
\begin{equation}\label{eq:cone_integral_computation_clear}
\int_{C_v}p(y)e^{z\cdot y}\,dy=e^{z\cdot v}|\det(u_{v,1},\cdots,u_{v,m})| \int_{\mathbb R_+^{n-1}}p\left(v+\sum_{k=1}^{n-1} s_k u_{v,k}\right)e^{\sum_{k=1}^{n-1} s_k z\cdot u_{v,k}}\,ds.
\end{equation}
Expand
\begin{equation*}
p\left(v+\sum_{k=1}^{n-1} s_k u_{v,k}\right)
=\sum_{\alpha\in\mathbb N_0^{n-1}}c_\alpha^{(v)}s^\alpha .
\end{equation*}
Substituting this into \eqref{eq:cone_integral_computation_clear} and using Fubini's theorem,
\begin{equation*}
\int_{C_v}p(y)e^{z\cdot y}\,dy=e^{z\cdot v}|\det(u_{v,1},\ldots,u_{v,n-1})|
\sum_\alpha c_\alpha^{(v)}\prod_{k=1}^{n-1}\int_0^\infty 
s_k^{\alpha_k}e^{s_k z\cdot u_{v,k}}\,ds_k .
\end{equation*}
It follows from 
\begin{equation*}
\int_0^\infty s^\ell e^{as}\,ds=\frac{(-1)^{\ell+1}\ell!}{a^{\ell+1}},\quad \operatorname{Re}a<0, \quad  \ell \geq 1,
\end{equation*}
that in the convergence region,
\begin{equation*}
\int_{C_v}p(y)e^{z\cdot y}\,dy=e^{z\cdot v}R_{D,p,v}(z),
\end{equation*}
where
\begin{equation*}
R_{D,p,v}(z)=|\det(u_{v,1},\ldots,u_{v,n-1})|\sum_\alpha c_\alpha^{(v)}
\frac{(-1)^{|\alpha|+n-1}\alpha!} {\prod_{k=1}^{n-1}(z\cdot u_{v,k})^{\alpha_k+1}} .
\end{equation*}
The right-hand side is rational in \(z\), so it gives the meromorphic continuation
\begin{equation*}
\int_{C_v}^{\mathrm{mer}}p(y)e^{z\cdot y}\,dy=e^{z\cdot v}R_{D,p,v}(z).
\end{equation*}

\noindent {Step 3.}
For a tangent cone \(C_v\), the integral $\int_{C_v} p(y)e^{z\cdot y}\,dy$
converges absolutely when $\operatorname{Re}(z\cdot u_{v,k})<0$, $k=1,\cdots,n-1$.
Let \(\int_{C_v}^{\mathrm{mer}}\) denote its meromorphic continuation; in the
convergence region it agrees with the ordinary integral, and elsewhere it is
defined by the rational expression from Step 2.

Write $p(y)=\sum_\beta a_\beta y^\beta$ and $p(\partial_z):=\sum_\beta a_\beta \partial_z^\beta$. Barvinok's unweighted formula \cite[Theorem~2.6]{barvinok1993computing} gives
\begin{equation}\label{eq:unweighted_BB_meromorphic_identity}
\int_D e^{z\cdot y}\,dy=\sum_{v\in V(D)}\int_{C_v}^{\mathrm{mer}} e^{z\cdot y}\,dy.
\end{equation}
Since \(D\) is compact, differentiating under the integral sign yields
\begin{equation*}
p(\partial_z)\int_D e^{z\cdot y}\,dy=\int_D p(y)e^{z\cdot y}\,dy.
\end{equation*}
For fixed \(v\), the same holds in the convergence region, so
\begin{equation*}
p(\partial_z)\int_{C_v} e^{z\cdot y}\,dy=\int_{C_v} p(y)e^{z\cdot y}\,dy.
\end{equation*}
Both sides continue meromorphically, hence
\begin{equation*}
p(\partial_z)\int_{C_v}^{\mathrm{mer}} e^{z\cdot y}\,dy
=\int_{C_v}^{\mathrm{mer}} p(y)e^{z\cdot y}\,dy.
\end{equation*}
Applying \(p(\partial_z)\) to \eqref{eq:unweighted_BB_meromorphic_identity}
and using this termwise gives
\begin{equation*}
\int_D p(y)e^{z\cdot y}\,dy=\sum_{v\in V(D)}
\int_{C_v}^{\mathrm{mer}} p(y)e^{z\cdot y}\,dy.
\end{equation*}
By Step 2,
\begin{equation*}
\int_D p(y)e^{z\cdot y}\,dy=\sum_{v\in V(D)} e^{z\cdot v}R_{D,p,v}(z),
\end{equation*}
which is \eqref{eq:weighted_BB_expansion}.

The terms \(e^{z\cdot v}R_{D,p,v}(z)\) may have poles on \(z\cdot u_{v,k}=0\), but the full sum is entire, since it equals \(z\mapsto \int_D p(y)e^{z\cdot y}\,dy\). Thus the identity is meromorphic and all poles cancel in the sum.

\noindent {Step 4.}
It remains to prove that if \(p\not\equiv0\), then $R_{D,p,v}\not\equiv0$,
$\forall \,v\in V(D)$. Suppose \(R_{D,p,v}\equiv0\) for some vertex \(v\). Set
$X_k:=z\cdot u_{v,k}$, $k=1,\cdots,n-1$. Since \(u_{v,1},\cdots,u_{v,n-1}\) are linearly independent, \(X_1,\cdots,X_{n-1}\) are linear coordinates on \(\mathbb C^{n-1}\). Then
\begin{equation*}
0\equiv\sum_\alpha c_\alpha^{(v)} \frac{(-1)^{|\alpha|+n-1}\alpha!}
{X_1^{\alpha_1+1}\cdots X_{n-1}^{\alpha_{n-1}+1}}.
\end{equation*}
Choose \(M\) larger than every component of every multi-index \(\alpha\) in the
finite sum, and multiply by \(X_1^{M+1}\cdots X_{n-1}^{M+1}\). This gives the
polynomial identity
\begin{equation*}
0\equiv\sum_\alpha c_\alpha^{(v)}(-1)^{|\alpha|+n-1}\alpha!
X_1^{M-\alpha_1}\cdots X_{n-1}^{M-\alpha_{n-1}}.
\end{equation*}
By linear independence of monomials, $c_\alpha^{(v)}=0$, $\forall \,\alpha$.
Hence
\begin{equation*}
p\left(v+\sum_{k=1}^{n-1} s_k u_{v,k}\right)\equiv 0,
\end{equation*}
as a polynomial in \(s\). Since $s\longmapsto v+\sum_{k=1}^{n-1} s_k u_{v,k}$ is an invertible affine change of coordinates, this implies \(p\equiv0\), contradicting the assumption. Therefore \(R_{D,p,v}\not\equiv0\) for every vertex \(v\).
\end{proof}

%%%%%%%%%%%%%%%%%%%%%%%%%%%%%%%%%%%%%%%%%%%%%%%%%

\subsection{Generic rays and rational coefficients}

The next lemmas close the main technical gap in the peeling argument. It is not
enough that a nonzero rational function be nonzero on a generic ray in
\(\mathbb C^{n-1}\): the proof requires the purely imaginary case \(\gamma=0\) of
the physical ray family
\begin{equation*}
z=\rho(\gamma-i\omega),\quad \rho>0,\quad \omega\in\mathcal O_U.
\end{equation*}
We keep the slightly more general \(\gamma\)-dependent statement, since its
proof is identical and it records the natural physical ray family.
\begin{lemma}\label{lem:homogeneous_polynomial_physical_ray}
Let \(P\) be a nonzero homogeneous polynomial on \(\mathbb C^{n-1}\), let
\(\gamma\in\mathbb R^{n-1}\), \(\sigma\in\{1,-1\}\), and let
\(\mathcal O\subset\mathbb S^{n-2}\) be nonempty and relatively open. Then
\(P(\gamma-i\sigma\omega)\) is not identically zero on \(\mathcal O\). Hence
\begin{equation*}
\{\omega\in\mathcal O:P(\gamma-i\sigma\omega)=0\},
\end{equation*}
has empty interior in \(\mathcal O\).
\end{lemma}

\begin{proof}
Assume, to the contrary, that \(P(\gamma-i\sigma\omega)=0\) on a nonempty
relatively open subset \(\mathcal O_0\subset\mathcal O\). By homogeneity,
\begin{equation*}
P(|\xi|\gamma-i\sigma\xi)=0,
\end{equation*}
for all \(\xi=\rho\omega\), \(\rho>0\), \(\omega\in\mathcal O_0\), hence on an
open cone in \(\mathbb R^{n-1}\setminus\{0\}\).

The map
\begin{equation*}
\xi\longmapsto |\xi|\gamma-i\sigma\xi,
\end{equation*}
parametrizes a real-analytic maximally totally real \((n-1)\)-submanifold of
\(\mathbb C^{n-1}\). Indeed, it is injective since its imaginary part
\(-\sigma\xi\) determines \(\xi\). At \(\xi\) with \(e=\xi/|\xi|\), its
differential is $a\longmapsto \gamma(e\cdot a)-i\sigma a$. If this equals \(i(\gamma(e\cdot b)-i\sigma b)\) for some \(b\), then
\begin{equation*}
\sigma b=\gamma(e\cdot a),\quad -\sigma a=\gamma(e\cdot b).
\end{equation*}
Thus \(a=-\gamma(e\cdot\gamma)(e\cdot a)\), and taking the scalar product
with \(e\) gives
\begin{equation*}
\bigl(1+(e\cdot\gamma)^2\bigr)(e\cdot a)=0.
\end{equation*}
Hence \(e\cdot a=0\), so \(a=b=0\). Therefore the tangent space is totally
real, and the submanifold is a holomorphic uniqueness set.

Since \(P\) is entire and vanishes on a nonempty open piece of this
submanifold, \(P\equiv0\), contradicting the assumption. Hence
\(P(\gamma-i\sigma\omega)\not\equiv0\) on \(\mathcal O\), and its zero set
has empty interior in \(\mathcal O\).
\end{proof}

%%%

The preceding lemma lets us avoid zeros of a homogeneous numerator along a
physical ray. We next identify the leading power of the rational coefficient at
one vertex; this yields the two-sided estimate on a generic ray.

\begin{lemma}\label{lem:vertex_coefficient_asymptotics}
Let \(D\), \(p\), \(v\), and \(R_{D,p,v}\) be as in
Lemma~\ref{lem:weighted_BB_vertex_expansion}, and assume \(p\not\equiv0\).
Write
\begin{equation*}
p\left(v+\sum_{k=1}^{n-1} s_k u_{v,k}\right)=\sum_{\alpha} c_\alpha^{(v)} s^\alpha,
\end{equation*}
and set $\kappa_{p,v}:=\min\{|\alpha|:c_\alpha^{(v)}\ne0\}$. Then
\begin{equation}\label{eq:R_vertex_asymptotic}
R_{D,p,v}(\rho\lambda)=\rho^{-(n-1)-\kappa_{p,v}}
R_{D,p,v}^{[\kappa]}(\lambda)+O(\rho^{-(n-1)-\kappa_{p,v}-1}),\quad \mbox{as~} \rho\to+\infty,
\end{equation}
uniformly on compact subsets of $\left\{\lambda:\prod_{k=1}^{n-1}(\lambda\cdot u_{v,k})\ne0\right\}$. Here
\begin{equation}\label{eq:R_leading_homogeneous_rational}
R_{D,p,v}^{[\kappa]}(\lambda)
=|\det(u_{v,1},\ldots,u_{v,n-1})|\sum_{|\alpha|=\kappa_{p,v}}
c_\alpha^{(v)}\frac{(-1)^{\kappa_{p,v}+n-1}\alpha!}
{\prod_{k=1}^{n-1}(\lambda\cdot u_{v,k})^{\alpha_k+1}},
\end{equation}
is a nonzero homogeneous rational function of degree
\(-(n-1)-\kappa_{p,v}\).
\end{lemma}

\begin{proof}
Write \(\kappa=\kappa_{p,v}\).  The expansion
\eqref{eq:R_vertex_asymptotic} follows from
\eqref{eq:explicit_R_Dpv} by setting \(z=\rho\lambda\) and grouping terms
according to \(|\alpha|\).  The terms with \(|\alpha|=\kappa\) give the
leading order \(\rho^{-(n-1)-\kappa}\), while all remaining terms are
\(O(\rho^{-(n-1)-\kappa-1})\).

It remains to show that \(R_{D,p,v}^{[\kappa]}\not\equiv0\).  In the
coordinates \(X_k=\lambda\cdot u_{v,k}\), the expression
\eqref{eq:R_leading_homogeneous_rational} becomes
\begin{equation*}
\sum_{|\alpha|=\kappa}c_\alpha^{(v)} \frac{(-1)^{\kappa+n-1}\alpha!}
{X_1^{\alpha_1+1}\cdots X_{n-1}^{\alpha_{n-1}+1}}.
\end{equation*}
If \(R_{D,p,v}^{[\kappa]} \equiv 0\), multiplying by
\(X_1^{\kappa+1}\cdots X_{n-1}^{\kappa+1}\) would give
\begin{equation*}
\sum_{|\alpha|=\kappa}
c_\alpha^{(v)}(-1)^{\kappa+n-1}\alpha!
\prod_{k=1}^{n-1} X_k^{\kappa-\alpha_k}=0.
\end{equation*}
Since \(\alpha_k\le\kappa\), these are distinct monomials.  Hence all
their coefficients vanish, so \(c_\alpha^{(v)}=0\) for every
\(|\alpha|=\kappa\), contradicting the definition of \(\kappa\).
Therefore \(R_{D,p,v}^{[\kappa]}\not\equiv0\).
\end{proof}

%%%%

We now choose one direction that works for every vertex and every active
segment. Only finitely many polar hyperplanes, numerator zero sets, and
coincident vertex phases need to be excluded. 

\begin{lemma}\label{lem:finite_family_generic_ray}
Let \(\mathcal O\subset\mathbb S^{n-2}\) be nonempty and relatively open, and let $\{(\gamma_\ell,\sigma_\ell,D_\ell,p_\ell,v_\ell)\}_{\ell=1}^M$ be a finite family with $\gamma_\ell\in\mathbb R^{n-1}$, $\sigma_\ell\in\{\pm1\}$, \(D_\ell\) a compact simple convex polytope, \(p_\ell\not\equiv0\), and \(v_\ell\in V(D_\ell)\). Then there exists \(\omega\in\mathcal O\) such that,
for every \(\ell\),
\begin{equation*}
(\gamma_\ell-i\sigma_\ell\omega)\cdot u_{\ell,v_\ell,k}\ne0,
\quad k=1,\cdots,n-1, \quad R_{D_\ell,p_\ell,v_\ell}^{[\kappa]}
(\gamma_\ell-i\sigma_\ell\omega)\ne0,
\end{equation*}
and
\begin{equation*}
\omega\cdot(v-w)\ne 0, \quad  \forall \,   v\ne w\in V(D_\ell).
\end{equation*}
Consequently, for this \(\omega\), each function $\rho\mapsto
R_{D_\ell,p_\ell,v_\ell}
\bigl(\rho(\gamma_\ell-i\sigma_\ell\omega)\bigr)$ 
has no poles for all sufficiently large \(\rho\) and grows polynomially in
\(\rho\). In fact, for each \(\ell\), there exist constants
\(c_\ell,C_\ell>0\) and \(\rho_\ell>0\) such that
\begin{equation*}
c_\ell\rho^{-(n-1)-\kappa_\ell} \le \left|R_{D_\ell,p_\ell,v_\ell}
\bigl(\rho(\gamma_\ell-i\sigma_\ell\omega)\bigr)\right|
\le C_\ell\rho^{-(n-1)-\kappa_\ell},\quad \rho\ge\rho_\ell,
\end{equation*}
where \(\kappa_\ell=\kappa_{p_\ell,v_\ell}\).
\end{lemma}

\begin{proof}
For each denominator factor, the map $\omega\mapsto(\gamma_\ell-i\sigma_\ell\omega)\cdot u_{\ell,v_\ell,k}$ 
is a nonzero linear function on \(\mathcal O\). By
Lemma~\ref{lem:homogeneous_polynomial_physical_ray}, its zero set has empty
interior in \(\mathcal O\).

Similarly, after putting \(R_{D_\ell,p_\ell,v_\ell}^{[\kappa]}\) over a common
denominator, its numerator is a nonzero homogeneous polynomial. Again by
Lemma~\ref{lem:homogeneous_polynomial_physical_ray}, its zero set in
\(\mathcal O\) has empty interior. There are only finitely many such numerator
and denominator conditions, and the sets $
\{\omega\in\mathcal O:\omega\cdot(v-w)=0\}$, $v\ne w$, are also finite in number and have empty interior. Their union is closed with empty interior, so it cannot cover the Baire space \(\mathcal O\). Hence we may choose \(\omega\) outside this union, giving all required nonvanishing and distinct-phase conditions.

For this \(\omega\), each restricted denominator is nonzero for \(\rho>0\), so
the functions
\begin{equation*}
\rho\mapsto R_{D_\ell,p_\ell,v_\ell}\bigl(\rho(\gamma_\ell-i\sigma_\ell\omega)\bigr),
\end{equation*}
have no poles. The two-sided estimate now follows from
\eqref{eq:R_vertex_asymptotic}--\eqref{eq:R_leading_homogeneous_rational} and
the nonvanishing of the leading coefficient at
\(\gamma_\ell-i\sigma_\ell\omega\).
\end{proof}

%%%

The final ingredient separates a finite sum of oscillatory terms. The next
lemma rules out exponential decay when at least one Laurent coefficient is
nonzero.
\begin{lemma}\label{lem:dominant_oscillatory_laurent_separation}
Let \(\beta_1,\cdots,\beta_M\in\mathbb R\) be pairwise distinct, and let
\(r_1,\cdots,r_M\) be nonzero finite Laurent polynomials in \(\rho\). If
\(\delta>0\), then
\begin{equation}
\label{eq:oscillatory_laurent_exponential_smallness}
\left|\sum_{k=1}^M e^{-i\rho\beta_k}r_k(\rho)\right|
\leq C e^{-\delta\rho}\rho^A,
\end{equation}
cannot hold for all sufficiently large \(\rho\).
\end{lemma}

\begin{proof}
Let \(d\) be the largest power of \(\rho\) occurring in any \(r_k\), and let
\(c_k\) be the coefficient of \(\rho^d\) in \(r_k\), with \(c_k=0\) if that
power is absent. At least one \(c_k\) is nonzero. Dividing
\eqref{eq:oscillatory_laurent_exponential_smallness} by \(\rho^d\) gives
\begin{equation*}
\sum_{k=1}^M c_k e^{-i\rho\beta_k}=o(1), \quad\text{as }\rho\to+\infty,
\end{equation*}
since \(r_k(\rho)/\rho^d=c_k+o(1)\). On the other hand, the frequencies
\(\beta_k\) are pairwise distinct, so
\begin{equation*}
\lim_{T\to+\infty}\frac1T\int_0^T\left|\sum_{k=1}^M c_k e^{-i\rho\beta_k}
\right|^2\,d\rho=\sum_{k=1}^M |c_k|^2>0.
\end{equation*}
But the integrand tends to \(0\) and is bounded, so its Ces\`aro mean must also
tend to \(0\). This contradiction proves the claim.
\end{proof}

%%%%%%%%%%%%%%%%%%%%%%%%%%%%%%%%%%%%%%%%%%%%

\subsection{Proof of the Brion--Barvinok injectivity theorem}

\begin{proof}[Proof of Theorem~\ref{thm:main_brion_barvinok_uniqueness}]
Assume that the first equation of \eqref{eq:BB_minus_identity_main} holds. By
\eqref{eq:BB_reduced_minus}, the reduced identity is
\begin{equation*}
\sum_{j=1}^N\frac{e^{-ir_j\cdot\zeta_\theta(\eta)}
-e^{-ir_{j-1}\cdot\zeta_\theta(\eta)}}{\lambda_j(\eta)}H_j(-i\eta)=0.
\end{equation*}
By the choice of \(\mathcal C_U\), the set \(U\cap\mathcal C_U\) is a nonempty
open subset of \(\mathcal C_U\). Since the left-hand side is real analytic on
\(\theta^\perp\setminus\{0\}\) and vanishes on \(U\cap\mathcal C_U\), analytic
continuation on the connected set \(\mathcal C_U\) shows that it vanishes on all
of \(\mathcal C_U\).

For each \(j\), Lemma~\ref{lem:weighted_BB_vertex_expansion} gives
\begin{equation*}
H_j(z)=\sum_{v\in V(D)}e^{z\cdot v}R_{j,v}(z),\quad R_{j,v}:=R_{D,p_j,v}.
\end{equation*}
Along a ray \(\eta=\rho\omega\), set $a_j(\omega):=1-i\gamma_j\cdot\omega$.
Since \(\operatorname{Re}a_j(\omega)=1\), we have \(a_j(\omega)\ne0\).
Substituting the vertex expansion into the reduced identity gives, away from
the polar hyperplanes,
\begin{equation}\label{eq:BB_vertex_identity_minus}
\sum_{j=1}^N\sum_{v\in V(D)}\frac{e^{\rho U_j(v;\omega)}
-e^{\rho L_j(v;\omega)}}{\rho a_j(\omega)}R_{j,v}(-i\rho\omega)=0.
\end{equation}
This identity is well-defined along every ray for which the finitely many
restricted denominators are not identically zero; after excluding their
finitely many positive zeros in \(\rho\), it holds for all sufficiently large
\(\rho\). It follows from the entire reduced identity, the right-hand side
being only a meromorphic representation of it.

Let $J:=\{j:p_j\not\equiv 0\}$. We prove \(J=\varnothing\). 

Suppose \(J\ne\varnothing\), and set $j_*:=\max J$. The upper endpoint of this segment has height \(t_{j_*}\), while every other nonzero endpoint contribution has height at most \(t_{j_*-1}\). Thus the exponential gap is
\begin{equation*}
\delta:=t_{j_*}-t_{j_*-1}=h_{j_*}>0.
\end{equation*}
Apply Lemma~\ref{lem:finite_family_generic_ray} to the family
\begin{equation*}
(0,1,D,p_j,v),\quad j\in J,\quad v\in V(D),
\end{equation*}
with \(\mathcal O=\mathcal O_U\). We obtain \(\omega\in\mathcal O_U\) such that all relevant vertex coefficients are nonzero finite Laurent polynomials along
$\eta=\rho\omega$, $z=-i\rho\omega$, and such that the numbers \(\omega\cdot v\), \(v\in V(D)\), are pairwise distinct.

The upper endpoint contribution of the \(j_*\)-th segment in \eqref{eq:BB_vertex_identity_minus} is
\begin{equation*}
e^{\rho t_{j_*}}
\sum_{v\in V(D)}
e^{-i\rho\omega\cdot(b_{j_*}+v)}
\frac{R_{j_*,v}(-i\rho\omega)}
{\rho a_{j_*}(\omega)}.
\end{equation*}
Each Laurent polynomial here is nonzero, and the frequencies
\(\omega\cdot(b_{j_*}+v)\) are pairwise distinct. All remaining terms in
\eqref{eq:BB_vertex_identity_minus} are bounded by $C e^{\rho(t_{j_*}-\delta)}\rho^A$
for suitable \(A,C>0\). Dividing the identity by \(e^{\rho t_{j_*}}\), we contradict
Lemma~\ref{lem:dominant_oscillatory_laurent_separation}. Hence \(J=\varnothing\), and therefore $p_1=\cdots=p_N=0$. This proves \(G=0\) a.e. in \(\Omega_{\mathcal B}\) for the negative phase.

For the sign-reversed case, assume that the second equation of \eqref{eq:BB_minus_identity_main} holds. The reduced identity is \eqref{eq:BB_reduced_plus}. Away from the polar hyperplanes, the vertex expansion gives
\begin{equation*}
\sum_{j=1}^N\sum_{v\in V(D)}\frac{e^{-\rho L_j(v;\omega)}
-e^{-\rho U_j(v;\omega)}}{\rho a_j(\omega)}R_{j,v}(i\rho\omega)=0.
\end{equation*}
Suppose \(J\ne\varnothing\), choose \(j_*:=\min J\). The lower endpoint of this segment has exponent \(-t_{j_*-1}\), while every other nonzero endpoint contribution has exponent at most \(-t_{j_*}\); the gap is again \(h_{j_*}>0\). Apply Lemma~\ref{lem:finite_family_generic_ray} to
\begin{equation*}
(0,-1,D,p_j,v),\quad j\in J,\quad v\in V(D),
\end{equation*}
and choose \(\omega\in\mathcal O_U\) with pairwise distinct vertex projections. The dominant lower-endpoint cluster is
\begin{equation*}
e^{-\rho t_{j_*-1}}\sum_{v\in V(D)}e^{i\rho\omega\cdot(b_{j_*-1}+v)}
\frac{R_{j_*,v}(i\rho\omega)}{\rho a_{j_*}(\omega)}.
\end{equation*}
After division by \(e^{-\rho t_{j_*-1}}\), all other terms are exponentially smaller.
Lemma~\ref{lem:dominant_oscillatory_laurent_separation}, now with frequencies
\(-\omega\cdot(b_{j_*-1}+v)\), gives a contradiction. Hence \(J=\varnothing\), and all \(p_j\) vanish identically also for the positive phase.
\end{proof}

%%%%%%%%%%%%%%%%%%%%%%%%%%%%%%%%%%%%%%%%%%%%%%%%%%%%%%%%%%%%	

\subsection{Geometric finite-dimensional stability}
\label{subsec:polygonal_stability}

We derive a conditional finite-dimensional estimate depending on the sweep geometry, the polynomial degree, and the selected Fourier--Laplace samples.

For the \(j\)-th segment, set
\begin{equation*}
v_j:=\frac{\theta+\gamma_j}{\sqrt{1+|\gamma_j|^2}},
\quad\beta_j:=\angle(v_j,\theta)\in[0,\pi/2),\quad
\ell_j:=h_j\sqrt{1+|\gamma_j|^2}.
\end{equation*}
Thus \(\ell_j\) is the Euclidean length of the segment, \(\beta_j\) is its inclination from the cylinder axis, and
\begin{equation*}
h_j=\ell_j\cos\beta_j,\quad
\gamma_j=\widehat\gamma_j\tan\beta_j,\quad
\widehat\gamma_j:=\frac{\gamma_j}{|\gamma_j|},\quad \gamma_j\ne 0.
\end{equation*}
Define the cumulative axial and lateral displacements
\begin{equation*}
T_j:=\sum_{s=1}^j\ell_s\cos\beta_s=t_j,
\quad B_j:=\sum_{s=1}^j\ell_s \widehat\gamma_s\sin\beta_s=b_j,
\quad T_0=B_0=0.
\end{equation*}
For the turning angle \(\alpha_j:=\angle(v_j,v_{j+1})\) at the \(j\)-th internal vertex,
\begin{equation}\label{eq:turning_angle_formula}
\cos\alpha_j=\cos\beta_j\cos\beta_{j+1}
+\widehat\gamma_j\cdot\widehat\gamma_{j+1}\sin\beta_j\sin\beta_{j+1}.
\end{equation}
Thus the vectors \(B_j\) encode the turning-angle information.

Fix \(d\in\mathbb N_0\), and set
\begin{equation*}
\mathcal I_d := \{\alpha\in\mathbb N_0^{n-1}:|\alpha|\le d\},\quad
M_d:=\binom{n-1+d}{d},\quad q_d:=NM_d=N\binom{n-1+d}{d}.
\end{equation*}
Choose \(y_D\in D\), set
\begin{equation*}
R_D:=\sup_{y\in D}|y-y_D|>0,\quad
m_\alpha^D(y):=\left(\frac{y-y_D}{R_D}\right)^\alpha,
\end{equation*}
and write
\begin{equation*}
p_j(y)=\sum_{\alpha\in\mathcal I_d} c_{j,\alpha}m_\alpha^D(y).
\end{equation*}
This scaled monomial basis satisfies
\begin{equation}\label{eq:scaled_monomial_bound}
|m_\alpha^D(y)|\le 1,\quad y\in D,\quad \alpha\in\mathcal I_d.
\end{equation}

Remove the known translation factor and define
\begin{equation*}
(\mathcal T_{\mathcal B}G)(\eta):=e^{ix_0\cdot\zeta_\theta(\eta)}
\int_{\Omega_{\mathcal B}} G(x)e^{-ix\cdot\zeta_\theta(\eta)}\,dx.
\end{equation*}
For \(\eta=\rho\omega\), where \(\rho>0\) and
\(\omega\in\theta^\perp\cap\mathbb S^{n-1}\), put
\begin{equation*}
F_\alpha^D(\rho,\omega):=\int_D m_\alpha^D(y)e^{-i\rho\omega\cdot y}\,dy.
\end{equation*}
Then $|F_\alpha^D(\rho,\omega)|\le |D|$. The exact contribution of \(c_{j,\alpha}\) to
\(\mathcal T_{\mathcal B}G\) is
\begin{equation*}
A_{j,\alpha}^{\mathrm{geo}}(\rho,\omega):=\frac{e^{\rho T_j-i\rho\omega\cdot B_j}
-e^{\rho T_{j-1}-i\rho\omega\cdot B_{j-1}}}{\rho\left(1-i\widehat\gamma_j\cdot\omega\tan\beta_j
\right)}F_\alpha^D(\rho,\omega),
\end{equation*}
and
\begin{equation}\label{eq:stability_column_length_angle_form}
(\mathcal T_{\mathcal B}G)(\rho\omega)
=\sum_{j=1}^N\sum_{\alpha\in\mathcal I_d}
A_{j,\alpha}^{\mathrm{geo}}(\rho,\omega)c_{j,\alpha}.
\end{equation}
Every term above is determined by the base \(D\), the lengths
\(\ell_j\), the inclinations \(\beta_j\), the turning directions
\(\widehat\gamma_j\), and the sampling pair \((\rho,\omega)\).

Choose \(q_d\) sampling pairs
\begin{equation}\label{eq:geometric_sampling_set}
\Xi:=\{(\rho_k,\omega_k):1\le k\le q_d\},\quad \rho_k\omega_k\in U,
\end{equation}
and define the geometric sampling matrix
\begin{equation*}
\mathbf A_{\mathcal B,d,\Xi}^{\mathrm{geo}}
:=\bigl(A_{j,\alpha}^{\mathrm{geo}}(\rho_k,\omega_k)
\bigr)_{1\le k\le q_d,\,1\le j\le N,\,\alpha\in\mathcal I_d}.
\end{equation*}
The geometric non-cancellation determinant is
\begin{equation*}
\Delta_{\mathcal B,d,\Xi}^{\mathrm{geo}}
:=\left|\det \mathbf A_{\mathcal B,d,\Xi}^{\mathrm{geo}}\right|.
\end{equation*}
It is explicit in the quantities in
\eqref{eq:stability_column_length_angle_form}; in particular, the turning
angles enter through the cumulative lateral vectors \(B_j\).

%%%

\begin{theorem}\label{thm:polygonal_finite_dimensional_stability}
Fix a nonempty open \(U\Subset\theta^\perp\setminus\{0\}\). Let
\(p_j\in\Pi_d(D)\), \(j=1,\cdots,N\). There exist \(q_d\) sampling pairs
\(\Xi\) as in \eqref{eq:geometric_sampling_set} such that
\(\Delta_{\mathcal B,d,\Xi}^{\mathrm{geo}}>0\). For any such \(\Xi\), set
\begin{equation*}
\rho^*:=\max_{1\le k\le q_d}\rho_k,\quad 
L_\theta:=T_N=\sum_{j=1}^N\ell_j\cos\beta_j,\quad h^*:=\max_j(\ell_j\cos\beta_j),
\quad B_{\mathcal B,\Xi}^{\mathrm{geo}}
:=h^*|D|e^{\rho^*L_\theta}.
\end{equation*}
Then
\begin{equation}\label{eq:sampled_polygonal_stability}
\|G\|_{L^2(\Omega_{\mathcal B})}\le C_{\mathcal B,d,\Xi}^{\mathrm{geo}}
\left(\sum_{k=1}^{q_d}|(\mathcal T_{\mathcal B}G)(\rho_k\omega_k)|^2
\right)^{1/2},
\end{equation}
where
\begin{equation}\label{eq:explicit_determinant_stability_bound}
C_{\mathcal B,d,\Xi}^{\mathrm{geo}}
:=\sqrt{h^*|D|M_d}\,\frac{q_d^{(q_d+1)/2}\left(B_{\mathcal B,\Xi}^{\mathrm{geo}}\right)^{q_d-1}}{\Delta_{\mathcal B,d,\Xi}^{\mathrm{geo}}}.
\end{equation}
In particular,
\begin{equation*}
\|G\|_{L^2(\Omega_{\mathcal B})}\le \sqrt{q_d}\,
C_{\mathcal B,d,\Xi}^{\mathrm{geo}}\|\mathcal T_{\mathcal B}G\|_{L^\infty(U)}.
\end{equation*}
The same bounds hold for differences of two degree-\(d\) densities, giving a
deterministic noise-amplification bound.
\end{theorem}

\begin{proof}
By Theorem~\ref{thm:main_brion_barvinok_uniqueness}, the functions
\begin{equation*}
A_{j,\alpha}^{\mathrm{geo}}(\rho,\omega),\quad
1\le j\le N,\quad \alpha\in\mathcal I_d,
\end{equation*}
are linearly independent on
\(X=\{(\rho,\omega):\rho\omega\in U\}\). Hence we may choose
\(x_k=(\rho_k,\omega_k)\in X\), \(1\le k\le q_d\), such that the evaluation
matrix \(\mathbf A_{\mathcal B,d,\Xi}^{\mathrm{geo}}\) is invertible: otherwise
all evaluation rows would lie in a proper subspace, giving a nontrivial linear
combination vanishing on \(X\). Thus
\(\Delta_{\mathcal B,d,\Xi}^{\mathrm{geo}}>0\).

Let \(\mathbf c=(c_{j,\alpha})\). By \eqref{eq:scaled_monomial_bound},
\begin{equation*}
|p_j(y)|\le \sqrt{M_d} 
\left(\sum_{\alpha\in\mathcal I_d}|c_{j,\alpha}|^2\right)^{1/2}.
\end{equation*}
Since the sweep map has unit Jacobian,
\begin{equation*}
\|G\|_{L^2(\Omega_{\mathcal B})}^2
=\sum_{j=1}^N h_j\int_D |p_j(y)|^2\,dy\le h^*|D|M_d
\sum_{j=1}^N\sum_{\alpha\in\mathcal I_d}|c_{j,\alpha}|^2.
\end{equation*}
Therefore
\begin{equation}\label{eq:geometric_coefficient_to_density_bound}
\|G\|_{L^2(\Omega_{\mathcal B})}
\le \sqrt{h^*|D|M_d}\,\|\mathbf c\|_{\ell^2}.
\end{equation}

For an entry of \(\mathbf A_{\mathcal B,d,\Xi}^{\mathrm{geo}}\), the
longitudinal factor in \eqref{eq:stability_column_length_angle_form} equals
\begin{equation*}
e^{\rho T_{j-1}-i\rho\omega\cdot B_{j-1}}\int_0^{\ell_j\cos\beta_j}
e^{\tau\rho(1-i\tan\beta_j\widehat\gamma_j\cdot\omega)}\,d\tau.
\end{equation*}
Hence
\begin{equation*}
|A_{j,\alpha}^{\mathrm{geo}}(\rho_k,\omega_k)|
\le \ell_j\cos\beta_j\,|D|e^{\rho_kT_j}\le B_{\mathcal B,\Xi}^{\mathrm{geo}}.
\end{equation*}
By Hadamard's inequality, every \((q_d-1)\times(q_d-1)\) cofactor of
\(\mathbf A_{\mathcal B,d,\Xi}^{\mathrm{geo}}\) is bounded by
$q_d^{(q_d-1)/2}\bigl(B_{\mathcal B,\Xi}^{\mathrm{geo}}\bigr)^{q_d-1}$. Thus
\begin{equation*}
\|\mathbf A^{-1}\|_2 \le \|\mathbf A^{-1}\|_{\mathrm F}
=\frac{\|\operatorname{adj}\mathbf A\|_{\mathrm F}}{|\det\mathbf A|}
\le\frac{q_d^{(q_d+1)/2}\bigl(B_{\mathcal B,\Xi}^{\mathrm{geo}}\bigr)^{q_d-1}
}{\Delta_{\mathcal B,d,\Xi}^{\mathrm{geo}}}.
\end{equation*}
Applying this inverse bound to the sampled system gives
\begin{equation}\label{eq:cramer_geometric_coefficient_bound}
\|\mathbf c\|_{\ell^2} \le \frac{q_d^{(q_d+1)/2}
\bigl(B_{\mathcal B,\Xi}^{\mathrm{geo}}\bigr)^{q_d-1}
}{\Delta_{\mathcal B,d,\Xi}^{\mathrm{geo}}}
\left(\sum_{k=1}^{q_d}|(\mathcal T_{\mathcal B}G)(\rho_k\omega_k)|^2\right)^{1/2}.
\end{equation}
Combining \eqref{eq:geometric_coefficient_to_density_bound} and
\eqref{eq:cramer_geometric_coefficient_bound} proves
\eqref{eq:sampled_polygonal_stability}--\eqref{eq:explicit_determinant_stability_bound}.
The \(L^\infty(U)\) estimate follows by bounding the Euclidean norm of the
\(q_d\) samples by \(\sqrt{q_d}\) times their maximum.
\end{proof}

After removing the initial endpoint factor, the \(j\)-th segment contributes
\begin{equation*}
S_j(\rho,\omega):=\frac{e^{\rho\ell_j\cos\beta_j(1-i\tan\beta_j\widehat\gamma_j\cdot\omega)}-1}{\rho(1-i\tan\beta_j\widehat\gamma_j\cdot\omega)}.
\end{equation*}
It obeys the length-angle bound
\begin{equation*}
\ell_j\cos^2\beta_j
\le \frac{\cos\beta_j}{\rho}\bigl(e^{\rho\ell_j\cos\beta_j}-1\bigr)
\le |S_j(\rho,\omega)|
\le \ell_j\cos\beta_j\,e^{\rho\ell_j\cos\beta_j}.
\end{equation*}
Indeed,
\begin{equation*}
|1-i\widehat\gamma_j\cdot\omega\tan\beta_j|\le \frac1{\cos\beta_j},
\quad |e^{\rho\ell_j\cos\beta_j(1-i\widehat\gamma_j\cdot\omega\tan\beta_j)}-1|
\ge e^{\rho\ell_j\cos\beta_j}-1,
\end{equation*}
which gives the middle inequality; \(e^s-1\ge s\) gives the first; and
\begin{equation*}
S_j(\rho,\omega)=\int_0^{\ell_j\cos\beta_j}
e^{\tau\rho(1-i\tan\beta_j\widehat\gamma_j\cdot\omega)}\,d\tau
\end{equation*}
gives the upper bound. Hence a short segment or \(\beta_j\approx\pi/2\) causes
a visible loss. The remaining inter-segment interaction is measured by
\(\Delta_{\mathcal B,d,\Xi}^{\mathrm{geo}}\).

%%%

\begin{remark}\label{rem:polygonal_stability_degeneracy}
Suppose $0<\ell_*\le\ell_j\le\ell^*$, $0\le\beta_j\le\beta^*<\frac{\pi}{2}$,
and fix sampling pairs with $\Delta_{\mathcal B,d,\Xi}^{\mathrm{geo}}\ge\Delta_*>0$.
Then
\begin{equation*}
\ell_*\cos^2\beta^*\le |S_j(\rho,\omega)|,\quad
h^*\le\ell^*,\quad L_\theta\le N\ell^*,
\end{equation*}
and hence
\begin{equation*}
C_{\mathcal B,d,\Xi}^{\mathrm{geo}} \le \sqrt{\ell^*|D|M_d}\,
\frac{q_d^{(q_d+1)/2}\bigl(\ell^*|D|e^{\rho^*N\ell^*}\bigr)^{q_d-1}}{\Delta_*}.
\end{equation*}
The turning angles enter \(\Delta_*\) through
\eqref{eq:turning_angle_formula} and the vectors \(B_j\). A lower bound on
the turning angles alone is not enough: the base moments, sampling directions,
and oscillatory phase cancellations must also remain nondegenerate.

Finally,
\begin{equation*}
M_d=\binom{n-1+d}{d},\quad
q_d=N\binom{n-1+d}{d}
\sim \frac{N}{(n-1)!}d^{\,n-1}.
\end{equation*}
Thus \eqref{eq:explicit_determinant_stability_bound} shows the rapid
deterioration as the polynomial degree increases.
\end{remark}

%%%%%%%%%%%%%%%%%%%%%%%%%%%%%%%%%%%%%%%%%%%%%%%%%%%%%%%%%%%%

\subsection{Approximation of a continuous sweep by polygonal models}
\label{subsec:curved_sweep_approximation}

We now use the finite-dimensional estimate as an approximation result. It does
not extend exact injectivity to all continuous densities; it gives an error
bound for a finite polygonal and polynomial model.

Let \(b\in C([0,L];\theta^\perp)\) with \(b(0)=0\) (which can be absorbed
into \(x_0\)). Define
\begin{equation*}
\Phi_b(y,t)=x_0+y+b(t)+t\theta, \quad (y,t)\in\operatorname{int}D\times(0,L),
\end{equation*}
and
\begin{equation*}
\Omega_b=\Phi_b(\operatorname{int}D\times(0,L)).
\end{equation*}
Since \((\Phi_b(y,t)-x_0)\cdot\theta=t\), \(\Phi_b\) is one-to-one. In
material coordinates,
\begin{equation*}
G_b(\Phi_b(y,t))=F(y,t),\quad F\in C(D\times[0,L]).
\end{equation*}
All approximation errors below are measured in these coordinates.

Let
\begin{equation*}
0=t_0<t_1<\cdots<t_N=L, \quad h_N=\max_{1\le j\le N}(t_j-t_{j-1}),
\end{equation*}
and let \(b_N\) be the piecewise linear interpolant of \(b\) at the nodes
\(t_j\). Put
\begin{equation*}
\delta_N=\|b-b_N\|_{L^\infty(0,L)}, \quad
\omega_b(r)=\sup_{\substack{t,s\in[0,L]\\ |t-s|\le r}}|b(t)-b(s)|.
\end{equation*}

%%%

The next lemma records the geometric approximation error.
\begin{lemma}
\label{lem:curve_interpolation}
The interpolant satisfies $\delta_N\le \omega_b(h_N)$, so \(\delta_N\to0\) as \(h_N\to0\). If \(b\in C^2([0,L];\theta^\perp)\), then
\begin{equation*}
\delta_N\le \frac18\|b''\|_{L^\infty(0,L)}h_N^2.
\end{equation*}
\end{lemma}

\begin{proof}
Fix \(t\in[t_{j-1},t_j]\) and write $t=(1-\lambda)t_{j-1}+\lambda t_j$, $0\le\lambda\le1$. Since \(b_N(t)=(1-\lambda)b(t_{j-1})+\lambda b(t_j)\),
\begin{equation*}
|b(t)-b_N(t)|\le (1-\lambda)|b(t)-b(t_{j-1})|+\lambda|b(t)-b(t_j)|
\le\omega_b(h_N).
\end{equation*}
Hence \(\delta_N\le\omega_b(h_N)\), and \(\delta_N\to0\) as \(h_N\to0\)
because \(b\) is uniformly continuous on \([0,L]\).

If \(b\in C^2\), the standard linear interpolation error bound on
\([t_{j-1},t_j]\) gives
\begin{equation*}
|b(t)-b_N(t)|\le\frac{(t_j-t_{j-1})^2}{8}\|b''\|_{L^\infty(t_{j-1},t_j)}.
\end{equation*}
Taking the maximum over all intervals yields
\begin{equation*}
\delta_N\le \frac18\|b''\|_{L^\infty(0,L)}h_N^2.
\end{equation*}
\end{proof}

The curve \(b_N\) produces the polygonal cylinder data used above.. We next
approximate the density. Define
\begin{equation*}
e_d(F):=\sup_{t\in[0,L]}\inf_{p\in\Pi_d(D)}\|F(\cdot,t)-p\|_{C(D)}, \quad
\omega_{F,t}(r):=\sup_{\substack{y\in D,\ t,s\in[0,L]\\ |t-s|\le r}}|F(y,t)-F(y,s)|.
\end{equation*}

%%%

\begin{lemma}\label{lem:piecewise_polynomial_approximation}
There exist \(p_j\in\Pi_d(D)\) such that $P_{N,d}(y,t)=p_j(y)$, $t_{j-1}<t\le t_j$,
satisfies
\begin{equation*}
\|F-P_{N,d}\|_{L^\infty(D\times(0,L))}\le E_{N,d}(F),
\quad E_{N,d}(F)=\omega_{F,t}(h_N)+e_d(F).
\end{equation*}
Moreover,
\begin{equation*}
E_{N,d}(F)\to 0, \quad\text{as } h_N\to0,\ d\to\infty.
\end{equation*}
\end{lemma}	

\begin{proof}
For each \(j\), choose \(t_j^*\in[t_{j-1},t_j]\). Since \(\Pi_d(D)\) is
finite-dimensional, a best approximant \(p_j\in\Pi_d(D)\) to
\(F(\cdot,t_j^*)\) exists, so
\begin{equation*}
\|F(\cdot,t_j^*)-p_j\|_{C(D)}\le e_d(F).
\end{equation*}
For \(t\in(t_{j-1},t_j]\),
\begin{equation*}
\|F(\cdot,t)-p_j\|_{C(D)}\le \|F(\cdot,t)-F(\cdot,t_j^*)\|_{C(D)}
+\|F(\cdot,t_j^*)-p_j\|_{C(D)}\le \omega_{F,t}(h_N)+e_d(F).
\end{equation*}
This proves the first assertion.

For the second, define
\begin{equation*}
a_d(t)=\inf_{p\in\Pi_d(D)}\|F(\cdot,t)-p\|_{C(D)}.
\end{equation*}
The map \(t\mapsto F(\cdot,t)\) is continuous into \(C(D)\), and the distance
to a fixed subspace is \(1\)-Lipschitz; hence each \(a_d\) is continuous. For
each fixed \(t\), Stone--Weierstrass gives \(a_d(t)\to0\). Since \([0,L]\) is
compact and \(a_d\) is continuous, Dini's theorem gives uniform convergence, so
\begin{equation*}
e_d(F)=\max_{t\in[0,L]}a_d(t)\to0.
\end{equation*}
Also \(\omega_{F,t}(h_N)\to0\) by uniform continuity of \(F\). Hence
\(E_{N,d}(F)\to0\).
\end{proof}

%%%

We now compare the exact and reduced Fourier--Laplace data. After removing the
known \(x_0\)-factor, define
\begin{equation*}
(\mathcal T_bF)(\eta)=\int_0^L\int_D
F(y,t)e^{-i\eta\cdot(y+b(t))}e^{|\eta|t}\,dy\,dt.
\end{equation*}
The same definition applies to \(\mathcal T_{b_N}P_{N,d}\). On each interval,
\(x=\Phi_{b_N}(y,t)\) has unit Jacobian, so \(\mathcal T_{b_N}P_{N,d}\) is exactly
the normalized polygonal transform \(\mathcal T_{\mathcal B_N}G_{N,d}\), where
\(G_{N,d}(\Phi_{b_N}(y,t))=P_{N,d}(y,t)\). For \(\rho\ge0\), set
\begin{equation*}
J_L(\rho)=
\begin{cases}
\dfrac{e^{\rho L}-1}{\rho},&\rho>0,\\[1ex]
L,&\rho=0.
\end{cases}
\end{equation*}

\begin{lemma}
\label{lem:curved_transform_model_error}
Let $M=\|F\|_{L^\infty(D\times(0,L))}$ and $E=E_{N,d}(F)$. Then, for every \(\eta\in\theta^\perp\) with \(\rho=|\eta|\),
\begin{equation}\label{eq:curved_transform_model_error}
\left|(\mathcal T_bF)(\eta)-(\mathcal T_{b_N}P_{N,d})(\eta)\right|
\le |D|J_L(\rho)\bigl[E+(M+E)\rho\delta_N\bigr].
\end{equation}
\end{lemma}

\begin{proof}
For each \((y,t)\),
\begin{equation*}
F e^{-i\eta\cdot b}
-
P_{N,d} e^{-i\eta\cdot b_N}
=
(F-P_{N,d})e^{-i\eta\cdot b}
+
P_{N,d}\bigl(e^{-i\eta\cdot b}-e^{-i\eta\cdot b_N}\bigr).
\end{equation*}
Since \(b,b_N\in\theta^\perp\), the exponential factors have modulus one,
and \(|e^{ia}-e^{ib}|\le |a-b|\) gives
\begin{equation*}
\bigl|e^{-i\eta\cdot b(t)}-e^{-i\eta\cdot b_N(t)}\bigr|
\le \rho\delta_N.
\end{equation*}
By Lemma~\ref{lem:piecewise_polynomial_approximation},
\(\|P_{N,d}\|_\infty\le M+E\). Multiplying by \(e^{\rho t}\) and integrating
over \(D\times(0,L)\), using \(\int_0^L e^{\rho t}\,dt=J_L(\rho)\), yields
\eqref{eq:curved_transform_model_error}.
\end{proof}

%%%

For the polygonal model \(\mathcal B_N\), let
\begin{equation*}
q_d=N\binom{n-1+d}{d}.
\end{equation*}
Fix nonempty open \(U\Subset\theta^\perp\setminus\{0\}\). Choose a sampling set
\begin{equation*}
\Xi_{N,d}=\{(\rho_k,\omega_k):1\le k\le q_d\},\quad \rho_k\omega_k\in U,
\end{equation*}
such that the matrix in Theorem~\ref{thm:polygonal_finite_dimensional_stability} is invertible. Set \(\eta_k=\rho_k\omega_k\), $k=1,\cdots,q_d$. For \(\varepsilon>0\), suppose the measured samples satisfy
\begin{equation}\label{eq:model}
d_k^\delta=(\mathcal T_bF)(\eta_k)+e_k, \quad k=1,\cdots,q_d,  \quad
e=(e_1,\cdots,e_{q_d})^{T}, \quad \|e\|_{\ell^2}\le\varepsilon.
\end{equation}
Let $\widehat P_{N,d}=\sum_{k=1}^{q_d}c_k\eta_k$, and let $\widehat{\mathbf c}=(c_1,\dots,c_{q_d})^{T}$ be the unique solution of the finite sampling system
\begin{equation*}
\mathbf A_{\mathcal B_N,d,\Xi_{N,d}}^{\mathrm{geo}}\widehat{\mathbf c}
=\mathbf d^\delta,
\end{equation*}
where \(\mathbf A_{\mathcal B_N,d,\Xi_{N,d}}^{\mathrm{geo}}\) is the discrete sampling
matrix for \(\mathcal T_bF\) and \(\mathbf d^\delta=(d_1^\delta,\dots,d_{q_d}^\delta)^T\). Define
\begin{equation*}
\mathcal M_{N,d,\Xi}=|D|
\left\{\sum_{k=1}^{q_d}J_L(\rho_k)^2\bigl[E+(M+E)\rho_k\delta_N\bigr]^2\right\}^{1/2}.
\end{equation*}

%%%

\begin{theorem}
\label{thm:curved_sweep_approximation_stability}
The reconstruction defined in (\ref{eq:model}) satisfies
\begin{equation}
\begin{aligned}\label{eq:curved_sweep_approximation_stability}
\|F-\widehat P_{N,d}\|_{L^2(D\times(0,L))}
\le (|D|L)^{1/2}E_{N,d}(F)+C_{\mathcal B_N,d,\Xi_{N,d}}^{\mathrm{geo}}
\left(\varepsilon+\mathcal M_{N,d,\Xi}\right).
\end{aligned}
\end{equation}
\end{theorem}

\begin{proof}
By the triangle inequality,
\begin{equation*}
\|F-\widehat P_{N,d}\|_{L^2}\le\|F-P_{N,d}\|_{L^2}+\|P_{N,d}-\widehat P_{N,d}\|_{L^2}.
\end{equation*}
Lemma~\ref{lem:piecewise_polynomial_approximation} gives
\begin{equation*}
\|F-P_{N,d}\|_{L^2(D\times(0,L))}\le (|D|L)^{1/2}E_{N,d}(F).
\end{equation*}
Since \(P_{N,d}\) and \(\widehat P_{N,d}\) lie in the same
finite-dimensional polygonal model,
Theorem~\ref{thm:polygonal_finite_dimensional_stability} yields
\begin{equation*}
\|P_{N,d}-\widehat P_{N,d}\|_{L^2} \le C_{\mathcal B_N,d,\Xi_{N,d}}^{\mathrm{geo}}
\|\mathcal T_{b_N}P_{N,d}-\mathbf d^\delta\|_{\ell^2}.
\end{equation*}
By Lemma~\ref{lem:curved_transform_model_error} and the noise assumption,
\begin{equation*}
\|\mathcal T_{b_N}P_{N,d}-\mathbf d^\delta\|_{\ell^2}
\le \|\mathcal T_{b_N}P_{N,d}-\mathcal T_bF\|_{\ell^2}+\|e\|_{\ell^2}
\le \mathcal M_{N,d,\Xi}+\varepsilon.
\end{equation*}
Combining these estimates proves
\eqref{eq:curved_sweep_approximation_stability}.
\end{proof}

The previous estimate yields an explicit rate once separate approximation bounds
are known. Suppose that, for some \(0<\alpha\le1\),
\begin{equation*}
\omega_{F,t}(r)\le C_F r^\alpha,\quad e_d(F)\le C_F d^{-\alpha},
\end{equation*}
and that \(b\in C^2([0,L];\theta^\perp)\). The second bound is a Jackson-type
hypothesis on \(D\). Then
\begin{equation*}
E_{N,d}(F)\le C_F(h_N^\alpha+d^{-\alpha}),\quad\delta_N\le \frac18\|b''\|_\infty h_N^2.
\end{equation*}
If \(\rho_*=\max_k\rho_k\), then
\begin{equation*}
\mathcal M_{N,d,\Xi}\le \sqrt{q_d}|D|L e^{\rho_*L}\left[E_{N,d}(F)
+\bigl(M+E_{N,d}(F)\bigr)\rho_*\frac18\|b''\|_\infty h_N^2\right].
\end{equation*}
This estimate alone does not imply convergence as \(N,d\to\infty\), since
\(C_{\mathcal B_N,d,\Xi_{N,d}}^{\mathrm{geo}}\) contains the reciprocal of the
sampling determinant. A convergent reconstruction requires
\begin{equation*}
E_{N,d}(F)\to 0 \quad\text{and}\quad C_{\mathcal B_N,d,\Xi_{N,d}}^{\mathrm{geo}}
\bigl(\varepsilon_{N,d}+\mathcal M_{N,d,\Xi}\bigr)\to0.
\end{equation*}
Under these conditions,
Theorem~\ref{thm:curved_sweep_approximation_stability} gives
\(\widehat P_{N,d}\to F\) in \(L^2(D\times(0,L))\).

\begin{remark}\label{rem:polynomial_not_analytic}
Theorem~\ref{thm:main_brion_barvinok_uniqueness} uses the polynomial assumption
through the finite exponential--rational expansion in Lemma~\ref{lem:weighted_BB_vertex_expansion}. Polynomial approximation does not
preserve an exact vanishing Fourier--Laplace identity, so exact uniqueness
remains restricted to polynomial densities. By contrast,
Theorem~\ref{thm:curved_sweep_approximation_stability} controls a finite
reconstruction of a continuous density and accounts explicitly for approximation
errors and the sampling condition number.
\end{remark}

%%%%%%%%%%%%%%%%%%%%%%%%%%%%%%%%%%%%%%%%%%%%%%%%%%%%%%%%%%%%
\section{Recovery of the Quotient in the Hyperbolic Inverse Problem}
\label{sec:hyperbolic_quotient_recovery}

Recall \(u_\iota\), \(\iota=1,2\), from Section~2.6 and set
\begin{equation*}
U_{\iota,\tau}:=\int_0^\infty e^{-\tau t}u_\iota(t,\cdot)\,dt.
\end{equation*}
We work in the Laplace domain, avoiding time differentiation of the sharply
truncated \(L^2\) initial states.

%%%%%%%%%%%%%%%%%%%%%%%%%%%%%%%

\subsection{A localized \texorpdfstring{\(L^p\)}{Lp} elliptic correction}

Choose a bounded smooth domain \(Q\) with \(\Omega_\sharp\Subset Q\), fix
\(U\Subset\theta^\perp\setminus\{0\}\), and set
\(\zeta_\theta(\xi)=\xi+i|\xi|\theta\).

\begin{lemma}
\label{lem:uniform_shifted_lattice_inverse}
Let \(1<p<\infty\). For \(\sigma=\pm1\), set $L_{\sigma,\xi}:=\Delta+2i\sigma \zeta_\theta(\xi)\cdot\nabla$. There exist right inverses $G_{\sigma,\xi}:L^p(Q)\to W^{2,p}(Q)$ such that
\begin{equation}\label{eq:right_inverse_equation}
L_{\sigma,\xi}G_{\sigma,\xi}F=F,\quad\text{in }Q,
\end{equation}
and
\begin{equation}\label{eq:shifted_inverse_estimate}
\|G_{\sigma,\xi}F\|_{W^{2,p}(Q)}\le C_{U,p}\|F\|_{L^p(Q)},
\quad\text{uniformly for }\xi\in U.
\end{equation}
\end{lemma}

\begin{proof}
Let \(\Delta_D^{-1}:L^p(Q)\to W^{2,p}(Q)\cap W_0^{1,p}(Q)\) be the Dirichlet
solution operator. The global \(W^{2,p}\) estimate gives
\begin{equation*}
\|\Delta_D^{-1}F\|_{W^{2,p}(Q)}\le C_{Q,p}\|F\|_{L^p(Q)}.
\end{equation*}
For \(\sigma=\pm1\), define
\begin{equation*}
G_{\sigma,\xi}F:=e^{-i\sigma \zeta_\theta(\xi)\cdot x}
\Delta_D^{-1}\bigl(e^{i\sigma \zeta_\theta(\xi)\cdot x}F\bigr).
\end{equation*}
Since \(\zeta_\theta(\xi)\cdot\zeta_\theta(\xi)=0\),
\begin{equation*}
(\Delta+2i\sigma \zeta_\theta(\xi)\cdot\nabla)
\bigl(e^{-i\sigma \zeta_\theta(\xi)\cdot x}v\bigr)
=e^{-i\sigma \zeta_\theta(\xi)\cdot x}\Delta v.
\end{equation*}
Hence \(L_{\sigma,\xi}G_{\sigma,\xi}F=F\). Multiplication by the exponentials
is uniformly bounded on \(L^p(Q)\) and \(W^{2,p}(Q)\) for \(\xi\in U\), so the
estimate follows from the Dirichlet bound.
\end{proof}

%%%

The right inverse gives uniform control of the conjugated equation. We now use
it in a contraction argument. The factor \(\tau^2\) makes the map contractive
and permits an \(L^\infty\) coefficient with a jump at the observation
boundary.

\begin{lemma}
\label{lem:hyperbolic_corrector}
Let \(q\in L^\infty(Q)\) with \(q\ge q_->0\), and let \(p_n\) satisfy
\eqref{eq:pn_condition}. There exist \(\tau_0>0\) and \(C_U>0\) such that,
for \(0<\tau<\tau_0\) and \(\xi\in U\), the equation
\begin{equation*}
\bigl(\Delta-2i\zeta_\theta(\xi)\cdot\nabla\bigr)\psi_{\tau,\xi}
=\tau^2 q(1+\psi_{\tau,\xi}), \quad\text{in }Q,
\end{equation*}
has a solution \(\psi_{\tau,\xi}\in W^{2,p_n}(Q)\) satisfying
\begin{equation}\label{eq:corrector_estimate}
\|\psi_{\tau,\xi}\|_{W^{2,p_n}(Q)}+\|\psi_{\tau,\xi}\|_{L^\infty(Q)}\le C_U\tau^2.
\end{equation}
\end{lemma}

\begin{proof}
By Lemma~\ref{lem:uniform_shifted_lattice_inverse}, solve the fixed-point
equation
\begin{equation*}
\psi=\mathcal T_{\tau,\xi}(\psi):=G_{-,\xi}\bigl[\tau^2q(1+\psi)\bigr],
\end{equation*}
in \(W^{2,p_n}(Q)\). Since \(2p_n>n\), we have
\(W^{2,p_n}(Q)\hookrightarrow L^\infty(Q)\), and therefore
\begin{equation*}
\|\mathcal T_{\tau,\xi}(\psi)\|_{W^{2,p_n}}
\le C_U\tau^2\|q(1+\psi)\|_{L^{p_n}}
\le C_U\tau^2\bigl(1+\|\psi\|_{W^{2,p_n}}\bigr),
\end{equation*}
while
\begin{equation*}
\|\mathcal T_{\tau,\xi}(\psi_1)-\mathcal T_{\tau,\xi}(\psi_2)\|_{W^{2,p_n}}
\le C_U\tau^2\|\psi_1-\psi_2\|_{W^{2,p_n}}.
\end{equation*}
Let \(C_0\) be a constant for which both estimates hold uniformly in
\(\xi\in U\). Choose \(\tau_0>0\) with \(2C_0\tau_0^2\le1\), and for
\(0<\tau<\tau_0\) set \(R_\tau=2C_0\tau^2\). If
\(\|\psi\|_{W^{2,p_n}}\le R_\tau\), then
\begin{equation*}
\|\mathcal T_{\tau,\xi}(\psi)\|_{W^{2,p_n}}
\le C_0\tau^2(1+R_\tau)\le 2C_0\tau^2 =R_\tau.
\end{equation*}
On this ball, \(\mathcal T_{\tau,\xi}\) is a contraction with Lipschitz
constant at most \(C_0\tau^2\le1/2\). Banach's fixed-point theorem gives a
unique fixed point, so
\begin{equation*}
\|\psi_{\tau,\xi}\|_{W^{2,p_n}}\le 2C_0\tau^2.
\end{equation*}
The Sobolev embedding, with a constant independent of \(\tau\) and \(\xi\),
yields \eqref{eq:corrector_estimate}. Only \(\|q\|_{L^\infty}\) is used, so a
jump of \(q\) across \(\partial\Omega_\sharp\) is allowed.
\end{proof}

For \(q=q_1\), define the spatial test function
\begin{equation*}
w_{\tau,\xi}(x):=e^{-ix\cdot\zeta_\theta(\xi)}\bigl(1+\psi_{\tau,\xi}(x)\bigr).
\end{equation*}
Then
\begin{equation*}
(-\Delta+\tau^2q_1)w_{\tau,\xi}=0,\quad\text{in }Q.
\end{equation*}

%%%%%%%%%%%%%%%%%%%%%%%%%%%%%%%%%%%%%%%%%%%%%%%%%%%%%%%%%%

\subsection{Laplace-domain integral identity and decoupling}

Set \(V_\tau:=U_{1,\tau}-U_{2,\tau}\). Equality of the Laplace-domain Cauchy
data gives
\begin{equation*}
\gamma_D V_\tau=0,\quad \gamma_N V_\tau=0, \quad\text{on }\partial\Omega_\sharp.
\end{equation*}
Also, \((-\Delta+\tau^2)V_\tau=0\) outside
\(\overline{\Omega}_\sharp\). By the zero Dirichlet trace and the energy
argument in Remark~\ref{rem:dirichlet_only_observation},
\begin{equation*}
V_\tau=0, \quad\text{in }\mathbb R^n\setminus\overline{\Omega}_\sharp.
\end{equation*}
By \eqref{eq:main_hyperbolic_laplace_equation},
\begin{equation*}
(-\Delta+\tau^2q_1)V_\tau =\tau(q_1f_1-q_2f_2)+\tau^2(q_2-q_1)U_{2,\tau},
\quad\text{in }Q,
\end{equation*}
with the right-hand side extended by zero outside \(\Omega_\sharp\).

Since \(U_{j,\tau},V_\tau\in H^2(\mathbb R^n)\), \(w_{\tau,\xi}\in W^{2,p_n}(Q)\),
and \(p_n>2\), we have \(w_{\tau,\xi}\in H^2(Q)\). Also
\(\overline{\Omega}_\sharp\Subset Q\), and \(V_\tau\) vanishes in a neighborhood
of \(\partial Q\), so its Dirichlet and Neumann traces on \(\partial Q\) are
zero. Hence the boundary term in Green's identity over \(Q\) vanishes. Dividing
by \(\tau\),
\begin{equation*}
\int_{\Omega_\sharp}(q_1f_1-q_2f_2)w_{\tau,\xi}\,dx
=-\tau\int_{\Omega_\sharp}(q_2-q_1)U_{2,\tau}w_{\tau,\xi}\,dx.
\end{equation*}
Equivalently,
\begin{align}\label{eq:hyperbolic_decoupling_tau}
\int_{\Omega_\sharp}(q_1f_1-q_2f_2)e^{-ix\cdot\zeta_\theta(\xi)}\,dx
&=-\int_{\Omega_\sharp}(q_1f_1-q_2f_2)e^{-ix\cdot\zeta_\theta(\xi)}
\psi_{\tau,\xi}\,dx \notag\\
&\quad -\tau\int_{\Omega_\sharp}(q_2-q_1)U_{2,\tau}
e^{-ix\cdot\zeta_\theta(\xi)}(1+\psi_{\tau,\xi})\,dx.
\end{align}
The first term on the right of (\ref{eq:hyperbolic_decoupling_tau}) satisfies
\begin{equation*}
\left|\int_{\Omega_\sharp}(q_1f_1-q_2f_2)e^{-ix\cdot\zeta_\theta(\xi)}
\psi_{\tau,\xi}\,dx\right|\le  C_U\tau^2\|q_1f_1-q_2f_2\|_{L^1(\Omega_\sharp)},
\end{equation*}
so it tends to zero. The assumed local decay gives
\begin{equation*}
\|U_{2,\tau}\|_{L^2(\Omega_\sharp)}
\le \int_0^\infty e^{-\tau t}\|u_2(t,\cdot)\|_{L^2(\Omega_\sharp)}\,dt
\le C\|\eta\|_{L^1}\|f_2\|_{L^2}.
\end{equation*}
Since \(e^{-ix\cdot\zeta_\theta(\xi)}(1+\psi_{\tau,\xi})\) is uniformly
bounded, Cauchy--Schwarz yields
\begin{equation*}
\left|\tau\int_{\Omega_\sharp}(q_2-q_1)U_{2,\tau}e^{-ix\cdot\zeta_\theta(\xi)}
(1+\psi_{\tau,\xi})\,dx\right|\le C_U\tau\|q_2-q_1\|_{L^\infty(\Omega_\sharp)}
|\Omega_\sharp|^{1/2}\|U_{2,\tau}\|_{L^2(\Omega_\sharp)}=O(\tau).
\end{equation*}
Letting \(\tau\to0^+\) yields
\begin{equation}\label{eq:hyperbolic_decoupled_complex_identity}
\int_{\Omega_\sharp}
(q_1f_1-q_2f_2)e^{-ix\cdot\zeta_\theta(\xi)}\,dx=0,\quad \xi\in U.
\end{equation}

%%%%%%%%%%%%%%%%%%%%%%%%%%%%%%%%%%%%%%%%%%%%%%%	

\subsection{Proof of the hyperbolic quotient theorem}

The weighted identity reduces the conclusion to the two Fourier--Laplace
injectivity theorems.

\begin{proof}[Proof of Theorem~\ref{thm:main_hyperbolic_quotient_uniqueness}]
Define
\begin{equation*}
G(x):=q_1(x)f_1(x)-q_2(x)f_2(x)=\frac{f_1(x)}{c_1^2(x)}-\frac{f_2(x)}{c_2^2(x)}.
\end{equation*}
By \eqref{eq:hyperbolic_decoupled_complex_identity},
\begin{equation*}
\int_{\Omega_\sharp}G(x)e^{-ix\cdot\zeta_\theta(\xi)}\,dx=0, \quad \xi\in U.
\end{equation*}

If \(\Omega_\sharp=\Omega_{\theta,\mu,h}\), then the oblique cylinder quotient
structure \eqref{eq:main_hyperbolic_quotient_structure_translated} gives
\begin{equation*}
G(x)=G_0(y)\chi_{\Omega_{\theta,\mu,h}}(x),
\quad G_0:=F_1-F_2\in L^1(D).
\end{equation*}
Since \(\Omega_{\theta,\mu,h}\) is the interval case of
\(\Omega_{\theta,\mu,E}\) with \(E=(0,h)\),
Theorem~\ref{thm:main_totally_real_uniqueness} yields
\(G=0\) a.e. in \(\Omega_\sharp\).

If \(\Omega_\sharp=\Omega_{\mathcal B}=\bigcup_{j=1}^N P_j\), then the
polygonal cylinder quotient structure
\eqref{eq:main_hyperbolic_quotient_structure_broken} gives
\begin{equation*}
G(x)=\sum_{j=1}^N G_j(Y(x))\chi_{P_j}(x),
\quad G_j:=F_{1,j}-F_{2,j}\in\mathbb C[y_1,\dots,y_{n-1}].
\end{equation*}
Since the available phase is \(e^{-ix\cdot\zeta_\theta(\xi)}\),
Theorem~\ref{thm:main_brion_barvinok_uniqueness} gives
\(G_1=\cdots=G_N=0\), hence \(G=0\) a.e. in \(\Omega_\sharp\).

In both cases,
\begin{equation*}
\frac{f_1}{c_1^2}=\frac{f_2}{c_2^2}, \quad\text{a.e. }\mathbb R^n,
\end{equation*}
since both quotients vanish outside \(\overline{\Omega}_\sharp\).
\end{proof}

%%%%%%%%%%%%%%%%%%%%%%%%%%%%%%%%%%%%%%%%%%%%%%%%%%%%%%%%%%%%%%%%

\section{Application to a Parabolic Passive Inverse Problem}
\label{sec:parabolic_application}

The proof of the parabolic passive inverse problem uses the same
Laplace-domain elliptic reduction as in the wave case.

%%%%%%%%%%%%%%%%%%%%%%%%%%%%%%%%%%%%%%%%%%

\subsection{Laplace-domain parabolic identity}

Let \(\lambda=\tau^2\) and define
\begin{equation*}
w_{\lambda,\xi}(x):=e^{-ix\cdot\zeta_\theta(\xi)}\bigl(1+\psi_{\tau,\xi}(x)\bigr),
\end{equation*}
where Lemma~\ref{lem:hyperbolic_corrector} is applied with \(q=q_1\). Then
\begin{equation*}
(-\Delta+\lambda q_1)w_{\lambda,\xi}=0,\quad\text{in }Q.
\end{equation*}

Set \(V_\lambda:=U_{1,\lambda}-U_{2,\lambda}\). Equality of the Laplace-domain
Cauchy data gives
\begin{equation*}
\gamma_D V_\lambda=0,\quad \gamma_N V_\lambda=0,
\quad\text{on }\partial\Omega_\sharp.
\end{equation*}
Since \((-\Delta+\lambda)V_\lambda=0\) outside
\(\overline{\Omega_\sharp}\), the same energy argument gives
\begin{equation*}
V_\lambda=0, \quad\text{in }\mathbb R^n\setminus\overline{\Omega}_\sharp.
\end{equation*}
By \eqref{eq:main_parabolic_laplace_equation},
\begin{equation*}
(-\Delta+\lambda q_1)V_\lambda =(q_1f_1-q_2f_2)+\lambda(q_2-q_1)U_{2,\lambda}, \quad\text{in }Q,
\end{equation*}
with the right-hand side extended by zero outside \(\Omega_\sharp\).

As above, \(V_\lambda\in H^2(\mathbb R^n)\) and
\(w_{\lambda,\xi}\in W^{2,p_n}(Q)\subset H^2(Q)\), since \(p_n>2\) and
\(|Q|<\infty\). Also \(\overline{\Omega}_\sharp\Subset Q\), so
\(V_\lambda\) vanishes near \(\partial Q\); hence its Dirichlet and Neumann
traces on \(\partial Q\) vanish. Green's identity over \(Q\) therefore gives
\begin{equation*}
\int_{\Omega_\sharp}
(q_1f_1-q_2f_2)w_{\lambda,\xi}\,dx=-\lambda\int_{\Omega_\sharp}
(q_2-q_1)U_{2,\lambda}w_{\lambda,\xi}\,dx.
\end{equation*}
Thus
\begin{align}\label{eq:parabolic_decoupling_tau}
\int_{\Omega_\sharp}(q_1f_1-q_2f_2)e^{-ix\cdot\zeta_\theta(\xi)}\,dx
&=-\int_{\Omega_\sharp}(q_1f_1-q_2f_2)e^{-ix\cdot\zeta_\theta(\xi)}
\psi_{\tau,\xi}\,dx\notag \\
&\quad -\lambda\int_{\Omega_\sharp} (q_2-q_1)U_{2,\lambda}e^{-ix\cdot\zeta_\theta(\xi)}(1+\psi_{\tau,\xi})\,dx.
\end{align}
Since \(\lambda=\tau^2\), the first term on the right of (\ref{eq:parabolic_decoupling_tau}) is bounded by $C_U\lambda\|q_1f_1-q_2f_2\|_{L^1(\Omega_\sharp)}$, and hence tends to zero.

For the second term on the right of (\ref{eq:parabolic_decoupling_tau}), set
$A(t):=\|u_2(t,\cdot)\|_{L^2(\Omega_\sharp)}$. Then
\begin{equation*}
\|U_{2,\lambda}\|_{L^2(\Omega_\sharp)} \le \int_0^\infty e^{-\lambda t}A(t)\,dt.
\end{equation*}
Since \(A\) is bounded and \(A(t)\to0\), the Abel mean satisfies
\begin{equation*}
\lambda\int_0^\infty e^{-\lambda t}A(t)\,dt\to 0,
\quad \mbox{as~} \lambda\to 0^+.
\end{equation*}
Indeed, given \(\varepsilon>0\), choose \(T\) with \(A(t)\le\varepsilon\) for
\(t\ge T\). Then
\begin{equation*}
\lambda\int_0^\infty e^{-\lambda t}A(t)\,dt
\le \lambda T\|A\|_{L^\infty(0,T)}+\varepsilon.
\end{equation*}
Using the uniform bound for the exponential factor and the corrector,
\begin{equation*}
\left|\lambda\int_{\Omega_\sharp}(q_2-q_1)U_{2,\lambda}e^{-ix\cdot\zeta_\theta(\xi)}
(1+\psi_{\tau,\xi})\,dx \right|
\le C_U\lambda\|q_2-q_1\|_{L^\infty(\Omega_\sharp)}|\Omega_\sharp|^{1/2}
\|U_{2,\lambda}\|_{L^2(\Omega_\sharp)}.
\end{equation*}
Hence the second term on the right of (\ref{eq:parabolic_decoupling_tau}) also tends to zero, and therefore
\begin{equation}\label{eq:para_decoupled_complex_identity}
\int_{\Omega_\sharp}(q_1f_1-q_2f_2)e^{-ix\cdot\zeta_\theta(\xi)}\,dx=0,
\quad \xi\in U.
\end{equation}

%%%%%%%%%%%%%%%%%%%%%%%%%%%%%%%%%%%%%%%%%%

\subsection{Proof of the parabolic quotient theorem}

We apply the transform injectivity results to the difference of the two
quotients.

\begin{proof}[Proof of Theorem~\ref{thm:main_parabolic_quotient_uniqueness}]
Define
\begin{equation*}
G(x):=q_1(x)f_1(x)-q_2(x)f_2(x)=\frac{f_1(x)}{\kappa_1(x)}
-\frac{f_2(x)}{\kappa_2(x)}.
\end{equation*}
By \eqref{eq:para_decoupled_complex_identity},
\begin{equation*}
\int_{\Omega_\sharp}G(x)e^{-ix\cdot\zeta_\theta(\xi)}\,dx=0,
\quad \xi\in U.
\end{equation*}

If \(\Omega_\sharp=\Omega_{\theta,\mu,h}\), then the oblique cylinder quotient
structure \eqref{eq:main_parabolic_quotient_structure_translated} gives
\begin{equation*}
G(x)=G_0(y)\chi_{\Omega_{\theta,\mu,h}}(x),
\quad G_0:=F_1-F_2\in L^1(D).
\end{equation*}
Applying Theorem~\ref{thm:main_totally_real_uniqueness} to the oblique cylinder
case \(E=(0,h)\), we get \(G=0\) a.e. \(\Omega_\sharp\).

If \(\Omega_\sharp=\Omega_{\mathcal B}=\bigcup_{j=1}^N P_j\), then the
polygonal cylinder quotient structure
\eqref{eq:main_parabolic_quotient_structure_broken} gives
\begin{equation*}
G(x)=\sum_{j=1}^N G_j(Y(x))\chi_{P_j}(x),
\quad G_j:=F_{1,j}-F_{2,j}\in\mathbb C[y_1,\dots,y_{n-1}].
\end{equation*}
Theorem~\ref{thm:main_brion_barvinok_uniqueness} applies to the same phase
\(e^{-ix\cdot\zeta_\theta(\xi)}\). Hence \(G_j=0\) for all \(j\), and
\(G=0\) a.e. \(\Omega_\sharp\).

In both cases,
\begin{equation*}
\frac{f_1}{\kappa_1}=\frac{f_2}{\kappa_2}, \quad\text{a.e. }\mathbb R^n,
\end{equation*}
since both quotients vanish outside \(\overline{\Omega}_\sharp\).
\end{proof}

%%%%%%%%%%%%%%%%%%%%%%%%%%%%%%%%%%%%%%%%%%%%%%%%%%%%%%%%%%%%

\section{Finite-Time Degradation of the Infinite-Time Result}
\label{sec:finite_time_stability}

This section proves Theorem~\ref{thm:main_finite_time_stability}. We first
separate the information contained in the measured boundary data from the
information that remains in the domain at time \(T\). 	

For each \(\iota=1,2\), let
\begin{equation*}
f_{\iota,k}\in D(A_{q_\iota}),\quad f_{\iota,k}\longrightarrow f_\iota,
\quad\text{in }L^2(\mathbb R^n),
\end{equation*}
and let \(u_{\iota,k}\) be the corresponding strong solution. Define
\begin{equation*}
W_{\iota,k}^T :=\int_0^T \eta(t)u_{\iota,k}(t)\,dt,
\end{equation*}
where $\eta(t)$ is given by (\ref{eq:WjT}).

\begin{lemma}\label{lem:finite_time_endpoint_identity}
For each \(\iota=1,2\), let $\mathcal{M}_{\iota,k}^{T} := \mathcal{G}_{W_{\iota,k}^{T}}$ \(\in \mathfrak{T}(\Omega_\sharp)^{*}\), then
\begin{equation}\label{eq:finite_time_boundary_extension}
\mathcal{M}_{\iota,k}^{T} \longrightarrow \mathcal{M}_{\iota}^{T},
\quad\text{in } \mathfrak{T}(\Omega_\sharp)^{*},
\end{equation}
and the limit is independent of the approximating sequence. Hence
\(\mathcal{M}_{\iota}^{T}\) is the unique continuous extension of the
weighted weak-Cauchy-trace functional from \(D(A_{q_\iota})\) to
\(L^{2}(\mathbb{R}^{n})\).

Moreover, for \(v \in \mathcal{H}(\Omega_\sharp)\) and generalized
finite-time boundary data as in
Subsection~\ref{subsec:main_finite_time}, both the wave and heat
equations satisfy
\begin{equation}\label{eq:finite_time_endpoint_identity}
\int_{\Omega_\sharp} G v \, dx=\int_{\Omega_\sharp} R_T v \, dx- \mathfrak{D}_T(v).
\end{equation}
If \(\mathcal{M}_{1}^{T} = \mathcal{M}_{2}^{T}\), then
\begin{equation*}
\left| \int_{\Omega_\sharp} G v \, dx \right|
\le \varepsilon_T \, \|v\|_{L^{2}(\Omega_\sharp)}.
\end{equation*}
\end{lemma}

\begin{proof}
Fix \(v\in\mathcal H(\Omega_\sharp)\). Recall the definition of \(W_\iota^T\) in (\ref{eq:WjT}); the Bochner integral lies in \(L^2(\Omega_\sharp)\).

For \(\iota=1,2\), choose \(f_{\iota,k}\in\mathcal D(A_{q_\iota})\) with
\begin{equation*}
f_{\iota,k}\longrightarrow f_\iota, \quad\text{in }L^2(\mathbb R^n),
\end{equation*}
and let {\(u_{\iota,k}\)} be the corresponding strong solution. Since {\(q_\iota\)}
is bounded above and below by positive constants, the weighted and
unweighted \(L^2\)-norms are equivalent.

\medskip
\noindent\emph{Case 1. Wave equation.}
By the spectral theorem,
\begin{equation*}
{\sup_{0\le t\le T}\|u_{\iota,k}(t)-u_\iota(t)\|_{L^2(\mathbb R^n)}
\le \|f_{\iota,k}-f_\iota\|_{L^2(\mathbb R^n)}},
\end{equation*}
hence
\begin{equation*}
{\|W_{\iota,k}^T-W_\iota^T\|_{L^2(\Omega_\sharp)}
\le \frac{T^2}{2} \sup_{0\le t\le T} \|u_{\iota,k}(t)-u_\iota(t)\|_{L^2(\Omega_\sharp)}}
\longrightarrow 0.
\end{equation*}
For the strong solution, spatial testing and two integrations in time give
\begin{equation*}
{\int_0^T(T-t)\partial_t^2u_{\iota,k}(t)\,dt
=\bigl[(T-t)\partial_tu_{\iota,k}(t)\bigr]_{0}^{T}+\int_0^T\partial_tu_{\iota,k}(t)\,dt
=u_{\iota,k}(T)-f_{\iota,k}},
\end{equation*}
since {\(\partial_tu_{\iota,k}(0)=0\)}. Therefore
\begin{equation*}
{\Delta W_{\iota,k}^T=q_\iota\bigl(u_{\iota,k}(T)-f_{\iota,k}\bigr)},
\end{equation*}
and the contraction estimate at \(t=T\) yields
\begin{equation*}
{\|\Delta W_{\iota,k}^T-q_\iota(u_\iota(T)-f_\iota)\|_{L^2(\Omega_\sharp)}
\le \|q_\iota\|_{L^\infty(\Omega_\sharp)}\Big(\|u_{\iota,k}(T)-u_\iota(T)\|_{L^2(\Omega_\sharp)}
+\|f_{\iota,k}-f_\iota\|_{L^2(\Omega_\sharp)}\Big)}\longrightarrow0.
\end{equation*}

\noindent\emph{Case 2. Heat equation.}
The semigroup contraction gives the same uniform convergence of
{\(u_{\iota,k}\) to \(u_\iota\)} on \([0,T]\). By the definitions of {\(W_{\iota,k}^T\)}
and {\(W_\iota^T\)},
\begin{equation*}
{\|W_{\iota,k}^T-W_\iota^T\|_{L^2(\Omega_\sharp)}
\le T\sup_{0\le t\le T} \|u_{\iota,k}(t)-u_\iota(t)\|_{L^2(\Omega_\sharp)}} \longrightarrow0,
\end{equation*}
and a single integration in time gives
\begin{equation*}
{\Delta W_{\iota,k}^T =q_\iota\int_0^T\partial_tu_{\iota,k}(t)\,dt=q_\iota\bigl(u_{\iota,k}(T)-f_{\iota,k}\bigr)},
\end{equation*}
so the semigroup contraction and the norm equivalence imply
\begin{equation*}
{\Delta W_{\iota,k}^T\longrightarrow q_\iota\bigl(u_\iota(T)-f_\iota\bigr)},
\quad\text{in }L^2(\Omega_\sharp).
\end{equation*}

In either model, the convergence of {\(W_{\iota,k}^T\)} and its Laplacian
gives
\begin{equation*}
{W_\iota^T\in H_\Delta(\Omega_\sharp),\quad \Delta W_\iota^T=q_\iota\bigl(u_\iota(T)-f_\iota\bigr),
\quad W_{\iota,k}^T\longrightarrow W_\iota^T},\quad \text{in }H_\Delta(\Omega_\sharp).
\end{equation*}
Indeed, the distributional identity for the limit follows by testing
against \(C_c^\infty(\Omega_\sharp)\) and passing to the limit in both
\(L^2\)-convergences.

Define {\(\mathcal M_{\iota,k}^T:=\mathcal G_{W_{\iota,k}^T}\)}. The continuity
estimate from Subsection~\ref{subsec:main_finite_time} gives
\begin{equation*}
{\|\mathcal M_{\iota,k}^T-\mathcal M_\iota^T\|_{\mathfrak T(\Omega_\sharp)^*}
\le C_n\bigl(\|W_{\iota,k}^T-W_\iota^T\|_{L^2(\Omega_\sharp)}
+\|\Delta W_{\iota,k}^T-\Delta W_\iota^T\|_{L^2(\Omega_\sharp)}\bigr)}\longrightarrow0,
\end{equation*}
with \(C_n\) independent of {\(\iota\)}, \(k\), and \(T\).

The same estimates bound the extension. Fix a coefficient \(q\) and set
\begin{equation*}
q_-:=\operatorname*{ess\,inf}_{\mathbb R^n}q>0,
\quad b_T:=
\begin{cases}
T^2/2, & \text{wave equation},\\
T, & \text{heat equation}.
\end{cases}
\end{equation*}
For \(f\in\mathcal D(A_q)\), the contraction estimates and the identity
\(\Delta W_f^T=q(u_f(T)-f)\) give
\begin{eqnarray*}
\|W_f^T\|_{L^2(\Omega_\sharp)}
&\le& b_T\,q_-^{-1/2}\|f\|_{L^2(\mathbb R^n)},\\
\|\Delta W_f^T\|_{L^2(\Omega_\sharp)}
&\le& 2\|q\|_{L^\infty(\mathbb R^n)}q_-^{-1/2}\|f\|_{L^2(\mathbb R^n)}.
\end{eqnarray*}
Consequently,
\begin{equation*}
\|\mathcal M_f^T\|_{\mathfrak T(\Omega_\sharp)^*}
\le C_n q_-^{-1/2}\bigl(b_T+2\|q\|_{L^\infty(\mathbb R^n)}\bigr)
\|f\|_{L^2(\mathbb R^n,q\,dx)}.
\end{equation*}
Thus, for each fixed \(T>0\), the map \(f\mapsto\mathcal M_f^T\) extends
uniquely and continuously from \(\mathcal D(A_q)\) to
\(L^2(\mathbb R^n,q\,dx)\).

For every \(k\), the strong solution lies in
{\(C([0,T];\mathcal D(A_{q_\iota}))\)}, and the graph norm on
{\(\mathcal D(A_{q_\iota})\)} is equivalent to the \(H^2(\mathbb R^n)\)-norm.
Hence its weak Dirichlet and Neumann traces are continuous in time in
\(H^{1/2}(\partial\Omega_\sharp)\) and
\(H^{-1/2}(\partial\Omega_\sharp)\), respectively. Since the bounded
trace maps commute with the \(H^2\)-valued Bochner integral, the weak
second Green identity and time integration give
\begin{eqnarray*}
{\mathcal M_{\iota,k}^T([\phi])}
=\int_0^T a_T(t)\Big(\langle\gamma_N{u_{\iota,k}}(t),\gamma_D\phi\rangle_{-1/2,1/2}
-\langle\gamma_N\phi,\gamma_D{u_{\iota,k}}(t)\rangle_{-1/2,1/2}\Big)\,dt.
\end{eqnarray*}
Thus the weighted weak Cauchy traces of the strong solutions determine
{\(\mathcal M_{\iota,k}^T\)}, and the convergence in
\(\mathfrak T(\Omega_\sharp)^*\) shows that {\(\mathcal M_\iota^T\)} is their
continuous extension. The limit is independent of the approximating
sequence, since it equals {\(\mathcal G_{W_\iota^T}\)}. This conclusion does
not assign pointwise-in-time Cauchy traces to the mild solution.

Now take the harmonic test function \(v\). Since \(\Delta v=0\), the
definition of {\(\mathcal M_\iota^T\)} and the terminal identity give
\begin{equation*}
{\mathcal M_\iota^T([v])=\int_{\Omega_\sharp}q_\iota\bigl(u_\iota(T)-f_\iota\bigr)v\,dx.}
\end{equation*}
Subtracting the identities for {\(\iota=1\)} and {\(\iota=2\)}, and using the
definitions of \(G\), \(R_T\), and \(\mathfrak D_T\), we get
\begin{equation*}
\mathfrak D_T(v)=\int_{\Omega_\sharp}(R_T-G)v\,dx,
\end{equation*}
which is equivalent to \eqref{eq:finite_time_endpoint_identity}.

If \(\mathcal M_1^T=\mathcal M_2^T\), then \(\mathfrak D_T(v)=0\), so
\begin{equation*}
\left|\int_{\Omega_\sharp}Gv\,dx\right|=\left|\int_{\Omega_\sharp}R_Tv\,dx\right|
\le \|R_T\|_{L^2(\Omega_\sharp)}\|v\|_{L^2(\Omega_\sharp)}
=\varepsilon_T\|v\|_{L^2(\Omega_\sharp)}.
\end{equation*}
This completes the proof.
\end{proof}

%%%

We first apply the endpoint identity to an oblique cylinder. The standard
null phase \(\eta+i|\eta|\theta\) suffices for uniqueness. A quantitative Fourier
estimate is easier using a phase adapted to the shear \(\mu\).

For \(\eta \in \theta^\perp \setminus \{0\}\), define
\begin{equation*}
\begin{aligned}
\chi_\mu(\eta)
&= \bigl((1+|\mu|^2)|\eta|^2-(\mu\cdot\eta)^2\bigr)^{1/2},
& \beta_\mu(\eta)=\frac{\mu\cdot\eta+i\chi_\mu(\eta)}{1+|\mu|^2},\\
\alpha_\mu(\eta)
&=\eta-\beta_\mu(\eta)\mu,
&\zeta_\mu(\eta)= \alpha_\mu(\eta)+\beta_\mu(\eta)\theta.
\end{aligned}
\end{equation*}

\begin{lemma}\label{lem:shear_adapted_null_phase}
The vector \(\zeta_\mu(\eta)\) satisfies $\zeta_\mu(\eta)\cdot\zeta_\mu(\eta)=0$.
Moreover, if $x=x_0+y+(\mu\cdot y+r)\theta$, $y\in D$, $0<r<h$, then
\begin{equation*}
e^{-i(x-x_0)\cdot\zeta_\mu(\eta)}=e^{-i\eta\cdot y}e^{-i\beta_\mu(\eta)r}.
\end{equation*}
\end{lemma}

\begin{proof}
Write \(\zeta=\zeta_\mu(\eta)\), \(\alpha=\alpha_\mu(\eta)\), and
\(\beta=\beta_\mu(\eta)\). Since \(\alpha,\mu,\eta\in\theta^\perp\)
and \(\theta\cdot\theta=1\),
\begin{equation*}
\zeta\cdot\zeta=|\alpha|^2+\beta^2=|\eta-\beta\mu|^2+\beta^2
=|\eta|^2-2\beta(\mu\cdot\eta)+(1+|\mu|^2)\beta^2.
\end{equation*}
The last expression vanishes because \(\beta\) solves
\begin{equation*}
(1+|\mu|^2)\beta^2-2(\mu\cdot\eta)\beta+|\eta|^2=0.
\end{equation*}
Hence \(\zeta\cdot\zeta=0\). Also, since \(\alpha+\beta\mu=\eta\),
\begin{equation*}
(x-x_0)\cdot\zeta=y\cdot\alpha+(\mu\cdot y+r)\beta=y\cdot\eta+r\beta.
\end{equation*}
Exponentiating gives the claimed factorization.
\end{proof}

%%%

The phase in the preceding lemma gives a direct estimate of the transverse
Fourier transform. We use the convention
\begin{equation*}
\widehat g(\eta)=\int_{\mathbb R^{n-1}}g(y)e^{-i\eta\cdot y}\,dy.
\end{equation*}

\begin{lemma}\label{lem:finite_time_fourier_bound}
Let \(G(x)=G_0(y)\chi_{\Omega_{\theta,\mu,h}}(x)\), where \(G_0\in L^2(D)\),
and let \(\widetilde{G}_0\) denote its extension by zero outside \(D\). Then
\begin{equation*}
|\widehat{\widetilde{G}}_0(\eta)|
\le C_{D,\mu,h}\,\varepsilon_T\,\langle\eta\rangle^{1/2}, \quad 
\eta\in\mathbb R^{n-1}.
\end{equation*}
\end{lemma}

\begin{proof}
For \(\eta\ne0\), set
\begin{equation*}
v_\eta(x)=e^{-i(x-x_0)\cdot\zeta_\mu(\eta)},
\quad \beta=\beta_\mu(\eta),
\quad b=b_\mu(\eta)=\frac{\chi_\mu(\eta)}{1+|\mu|^2}>0.
\end{equation*}
By Lemma~\ref{lem:shear_adapted_null_phase}, \(\Delta v_\eta=0\). Inserting
\(v_\eta\) into Lemma~\ref{lem:finite_time_endpoint_identity} and using the
unit Jacobian of \((y,r)\), we obtain
\begin{equation*}
B_h(\beta)\widehat{\widetilde{G}}_0(\eta)
=\int_{\Omega_{\theta,\mu,h}}R_T(x)v_\eta(x)\,dx,
\quad B_h(\beta)=\int_0^h e^{-i\beta r}\,dr.
\end{equation*}
Since
\begin{equation*}
|B_h(\beta)|=\left|\frac{e^{-i\beta h}-1}{-i\beta}\right|
\ge \frac{e^{bh}-1}{|\beta|}>0,
\end{equation*}
and
\begin{equation*}
\|v_\eta\|_{L^2(\Omega_{\theta,\mu,h})}^2
=|D|\int_0^h e^{2br}\,dr=|D|\frac{e^{2bh}-1}{2b},\quad b>0,
\end{equation*}
we have
\begin{equation*}
\frac{\|v_\eta\|_{L^2}^2}{|B_h(\beta)|^2}
\le|D|\frac{|\beta|^2}{2b}\coth\left(\frac{bh}{2}\right).
\end{equation*}
A direct calculation gives
\begin{equation*}
|\beta|^2=\frac{|\eta|^2}{1+|\mu|^2},\quad c_\mu|\eta|\le b\le C_\mu|\eta|,
\end{equation*}
for some constants \(c_\mu,C_\mu>0\). Since \(\coth s\le1+s^{-1}\) for
\(s>0\), it follows that
\begin{equation*}
\frac{\|v_\eta\|_{L^2}}{|B_h(\beta)|}\le C_{D,\mu,h}\langle\eta\rangle^{1/2}.
\end{equation*}
The Cauchy--Schwarz estimate in Lemma~\ref{lem:finite_time_endpoint_identity}
now proves the claim for \(\eta\ne0\).

Finally, since \(|D|<\infty\),
\begin{equation*}
\|G_0\|_{L^1(D)} \le |D|^{1/2}\|G_0\|_{L^2(D)},
\end{equation*}
so the zero extension \(\widetilde G_0\) lies in \(L^1(\mathbb R^{n-1})\)
and its Fourier transform is continuous. Letting \(\eta\to0\) in the estimate
for \(\eta\ne0\) gives the same bound at \(\eta=0\).
\end{proof}

%%%

We now complete the oblique cylinder part of the finite-time theorem.

\begin{proof}[Proof of \eqref{eq:main_finite_time_holder}]
Plancherel's theorem and
Lemma~\ref{lem:finite_time_fourier_bound} give
\begin{equation*}
\|\widetilde{G}_0\|_{H^r}^2 
\le C\varepsilon_T^2\int_{|\eta|\le N}\langle\eta\rangle^{2r+1}\,d\eta
+C\int_{|\eta|>N}\langle\eta\rangle^{2r}|\widehat{\widetilde{G}}_0(\eta)|^2\,d\eta
\le C\varepsilon_T^2N^{2r+n}+CM^2N^{-2(S-r)}.
\end{equation*}
Hence
\begin{equation*}
\|\widetilde{G}_0\|_{H^r}\le C\left(\varepsilon_TN^{r+n/2}+MN^{-(S-r)}\right).
\end{equation*}
If \(0<\varepsilon_T<M\), choose
\begin{equation*}
N=\left(\frac{M}{\varepsilon_T}\right)^{1/(S+n/2)}.
\end{equation*}
The two terms are then of the same order, and 
\begin{equation*}
\|\widetilde G_0\|_{H^r}\le CM^{1-\alpha_r}\varepsilon_T^{\alpha_r},
\quad \alpha_r=\frac{S-r}{S+n/2}.
\end{equation*}
If \(\varepsilon_T\ge M\), the same bound follows from
\(\|\widetilde G_0\|_{H^r}\le M\). If \(\varepsilon_T=0\),
Lemma~\ref{lem:finite_time_fourier_bound} gives
\(\widehat{\widetilde{G}}_0=0\), and the bound follows again. Thus
\eqref{eq:main_finite_time_holder} holds in all cases.

Finally, in the orthogonal coordinates
\(\theta^\perp\oplus\mathbb R\theta\), the map $(y,r)\longmapsto (y,\mu\cdot y+r)$
has unit Jacobian, and
\(G(x_0+y+(\mu\cdot y+r)\theta)=G_0(y)\). Hence
\begin{equation*}
\|G\|_{L^2(\Omega_{\theta,\mu,h})}^2=h\|G_0\|_{L^2(D)}^2.
\end{equation*}
Taking \(r=0\) gives the stated physical-space consequence.
\end{proof}

%%%

We finish with the polygonal cylinder estimate. This case is Lipschitz because the
unknown lies in a fixed finite-dimensional space.

\begin{proof}[Proof of \eqref{eq:main_finite_time_polygonal}]
For \(1\le k\le q_d\), set $v_k(x)=e^{-i(x-x_0)\cdot\zeta_\theta(\rho_k\omega_k)}$. Since \(\zeta_\theta\) is null, \(\Delta v_k=0\). By
Lemma~\ref{lem:finite_time_endpoint_identity} and the definition of the
normalized transform,
\begin{equation*}
(\mathcal T_{\mathcal B}G)(\rho_k\omega_k)=\int_{\Omega_{\mathcal B}}R_T(x)v_k(x)\,dx.
\end{equation*}
Cauchy--Schwarz gives
\begin{equation*}
|(\mathcal T_{\mathcal B}G)(\rho_k\omega_k)|
\le\varepsilon_T\|v_k\|_{L^2(\Omega_{\mathcal B})}.
\end{equation*}
Squaring and summing over \(k\), Theorem~\ref{thm:polygonal_finite_dimensional_stability}
yields
\begin{equation*}
\|G\|_{L^2(\Omega_{\mathcal B})} \le C_{\mathcal B,d,\Xi}^{\mathrm{geo}}
K_{\mathcal B,\Xi}\varepsilon_T.
\end{equation*}
This proves \eqref{eq:main_finite_time_polygonal}.
\end{proof}

The terminal residual is estimated by a local decay bound. Suppose
\begin{equation*}
\|u_\iota(T,\cdot)\|_{L^2(\Omega_\sharp)}\le M_0\varrho(T),
\quad \iota=1,2,
\end{equation*}
and \(\|q_\iota\|_{L^\infty}\le q_+\). Then
\begin{equation*}
\varepsilon_T\le2q_+M_0\varrho(T).
\end{equation*}
Hence the oblique cylinder estimate is of order
\(\varrho(T)^{(S-r)/(S+n/2)}\), while the polygonal cylinder estimate is of
order \(\varrho(T)\). This conclusion is conditional on the stated decay
bound; positivity and boundedness of the coefficient alone do not yield a
uniform decay rate.

%%%%%%%%%%%%%%%%%%%%%%%%%%%%%%%%%%%%%%%%%%%%%%%%%%%%%%%%%%%%%%%%%

\section{Conclusion}\label{sec:conclusion}

We proved injectivity of the null-phase Fourier--Laplace transform on two
classes of codimension-one densities. The oblique cylinder proof relies on
holomorphic uniqueness on a maximally totally real submanifold. The
polygonal cylinder proof uses a finite Brion--Barvinok expansion and separates
the ordered endpoints.

A fixed polygonal model of bounded polynomial degree satisfies a
finite-dimensional Lipschitz estimate. The same estimate controls the
reduction of a continuous generating curve and a continuous material density to a
polygonal and polynomial model. The reconstruction bound separates
measurement noise, density approximation, curve approximation, and the
sampling condition number. The last factor is essential: function and curve
approximation alone do not imply convergence of the inverse reconstruction.

For the wave and heat equations, the initial state vanishes outside the
observation domain. The Laplace-domain data determine the whole-space
quotients
\begin{equation*}
\frac{f}{c^2}\quad\text{and}\quad \frac{f}{\kappa},
\end{equation*}
respectively, so the coefficient and the initial state are not recovered
separately. If only the weighted generalized boundary data on \([0,T]\) are
measured, the quotient is determined up to the terminal residual. Equal finite-time data yield a H\"older estimate in the oblique cylinder model and a Lipschitz estimate in the fixed polygonal cylinder model. Both bounds contain
\(q_1u_1(T)-q_2u_2(T)\). Hence the infinite-time uniqueness result is
recovered when this residual vanishes, and a known decay estimate gives
approximate uniqueness as \(T\) increases. These estimates do not account for
noisy boundary data.

%%%%%%%%%%%%%%%%%%%%%%%%%%%%%%%%%%%%%%%%%%%%%

\section*{Acknowledgments}
Y. Jiang was supported in part by China's National Key R\&D Program (2024YFA1012302), and the RGC Project JRFS2627-1S06. H. Liu was supported by the Hong Kong RGC General Research Funds (projects 11311122, 11304224 and 11303125). T. Ni is supported by the Postdoctoral Fellowship Program (Grant No.BX2026004) and China Postdoctoral Science Foundation (Grant Number 2025M783082). The work of K. Zhang is supported in part by the National Natural Science Foundation of China (Grant No. 12271207).
%%%%%%%%%%%%%%%%%%%%%%%%%%%%%%%%%%%%%%%%%%%%%%%%%%%%%%%%%%%%%%%%%

\begingroup
\bibliographystyle{plain}
\bibliography{reference}
\endgroup
\end{document}